\documentclass{amsart}
\usepackage{amssymb}
\usepackage{color}
\usepackage[allcolors = blue, colorlinks = true]{hyperref}
\usepackage{mathptmx}
\usepackage{amsrefs}
\usepackage{comment}
\usepackage{mathtools}
\usepackage{esint}
\usepackage[
  hmarginratio={1:1},     % equal left and right margins
  vmarginratio={1:1},     % equal top and bottom margins
  textwidth=450pt,        % new text width
  heightrounded,          % always useful
]{geometry}
\newtheorem{theorem}{Theorem}[section]
\newtheorem{lemma}[theorem]{Lemma}
\newtheorem{corollary}[theorem]{Corollary}
\newtheorem{proposition}[theorem]{Proposition}
\theoremstyle{definition}
\newtheorem{example}[equation]{Example}
\theoremstyle{remark}
\newtheorem{remark}[equation]{Remark}
\newcommand{\R}{\mathbb{R}}
\newcommand{\Rn}{\R^n}

\newcommand{\Om}{\Omega}
\newcommand{\la}{\langle}
\newcommand{\ra}{\rangle}

\newcommand{\CcU}{C^{\infty}_0(U)}
\newcommand{\CcO}{C^{\infty}_0(\Omega)}
\newcommand{\loc}{\textnormal{loc}}
\newcommand{\vp}{\varphi}
\newcommand{\ue}{u^{\epsilon}}
\newcommand{\Ve}{V^{\epsilon}}
\newcommand{\ued}{u^{\epsilon,\delta}}
\newcommand{\nuk}{\nu_\kappa}
\newcommand{\muk}{\mu_\kappa}

\newcommand{\sqrtue}{\sqrt{|\nabla u^\varepsilon|^2+\varepsilon}}

\DeclareMathOperator{\tr}{tr\,}
\DeclareMathOperator{\diverg}{div\,}
\DeclareMathOperator{\dist}{dist\,}
\DeclareMathOperator{\dv}{div\,}

\numberwithin{equation}{section}

\begin{document}
%--------------------------------------------------------------------
\title[Local second order regularity of solutions to elliptic Orlicz-Laplace equation]{Local second order regularity of solutions to elliptic Orlicz-Laplace equation}

\author{Arttu Karppinen}
\address[Arttu Karppinen]{
Faculty of Mathematics, Informatics and Mechanics, University of Warsaw, ul. Banacha 2, 02-097 Warsaw, Poland}
\email{a.karppinen@uw.edu.pl}

\author{Saara Sarsa}
\address[Saara Sarsa]{
Department of Mathematics and Statistics, University of
Jyv\"askyl\"a, PO~Box~35, FI-40014 Jyv\"askyl\"a, Finland}
\email{saara.m.sarsa@jyu.fi}


\thanks{A. Karppinen was supported by NCN Grant Sonata Bis 2019/34/E/ST1/00120. 
S. Sarsa was supported by the Academy of Finland, project 308759, by the Academy of Finland, Center of Excellence in Randomness and Structures and by the Jenny and Antti Wihuri foundation.}

\begin{abstract}
We consider Orlicz--Laplace equation $-\diverg(\frac{\vp'(|\nabla u|)}{|\nabla u|}\nabla u)=f$ where $\vp$ is an Orlicz function and either $f=0$ or $f\in L^\infty$. We prove local second order regularity results for the weak solutions $u$ of the Orlicz--Laplace equation.
More precisely, we show that if $\psi$ is another Orlicz function that is close to $\vp$ in a suitable sense, then $\frac{\psi'(|\nabla u|)}{|\nabla u|}\nabla u\in W^{1,2}_{\loc}$. This work contributes to the building up of quantitative second order Sobolev regularity for solutions of nonlinear equations.
\end{abstract}
\maketitle
%\tableofcontents
%--------------------------------------------------------------------
\section{Introduction}
%--------------------------------------------------------------------
In this article we consider Orlicz--Laplace type equations
\begin{equation} \label{eq:OL-equation}
    -\diverg\Big(\frac{\vp'(|\nabla u|)}{|\nabla u|}\nabla u\Big)=f \quad\text{in }\Omega,
\end{equation}
where $\Omega \subset \Rn$, $n\geq 2$, is a bounded domain, $\vp$ is a $C^2$-regular Orlicz function that satisfies $(p,q)$ growth condition for $1<p\leq q<\infty$, and either $f=0$ or $f\in L^\infty(\Omega)$. Our main goal is to prove local second order regularity results for the weak solutions $u$ of \eqref{eq:OL-equation}. We refer the reader to Section~\ref{ssec:orlicz-spaces} for more details on definitions and assumptions.

Let us briefly describe the context surrounding the topic of this article. There are multiple types of second order regularity results in the literature. The most basic ones consider $L^2$-regularity of the Hessian matrix $D^2 u$ of the weak solutions of equations with quadratic growth. In the case of non-quadratic growth, there has been vast literature on regularity of nonlinear vector fields mirroring the equation in question.

To begin with, consider the case of $p$-Laplacian, i.e. $\vp(t)=\frac{1}{p}t^p$ for $1<p<\infty$. For the homogeneous problem $\Delta_p u=0$, the most natural second order regularity result states that $|\nabla u|^{\frac{p-2}{2}} \nabla u\in W^{1,2}_{\loc}(\Om;\Rn)$, see for instance \cite{Bojarski1987}*{Proposition 2} and \cite{Chen1989}. Such results also appear in \cite{Uhlenbeck1977,Lewis1983}, where they act as intermediate steps in the proof of the Hölder-continuity of the gradient $\nabla u$. Manfredi and Weitsman \cite{Manfredi1988} proved that $D^2u\in L^2_\loc(\Om)$ on the range $1<p<3+\frac{2}{n-2}$. These results were later developed by Dong, Peng, Zhang and Zhou \cite{Dong2020} and the second author \cite{Sarsa2022} to cover the $W^{1,2}_{\loc}$-regularity of a family of vector fields $|\nabla u|^\beta\nabla u$, where $\beta>-1+\frac{(n-2)(p-1)}{2(n-1)}$.

For the non-homogeneous problem $\Delta_pu=f$, a natural question to ask is the following: Given the regularity of $f$, can we trade the divergence operator $\diverg$ with the full derivative $D$? In other words, the most natural second order results concern the regularity of the vector field $|\nabla u|^{p-2}\nabla u$, also known as stress field. For these type of results, we refer to \cites{Avelin2018,Lou2008,Damascelli2004,Mingione2010}. For some other results for the non-homogeneous problem, see \cite{Simon1978,DeThelin1982,BeiraodaVeiga2013,Crispo2016,Cellina2017-23,Cellina2017-34,Montoro2023}. 

Many results for the vector fields $|\nabla u|^{\frac{p-2}{2}}\nabla u$ and $|\nabla u|^{p-2}\nabla u$ stemming from the prototypical $p$-Laplacian can be generalized to the case of non-standard growth, see for instance \cites{Diening2009,Challal2010,Esposito2002,Esposito2004,Carozza2011} and \cites{Cianchi2018,Cianchi2019,Balci2021,Miao2023,Antonini2023}, respectively. For similar results with $p(x)$-Laplace equation we refer to \cites{Challal2011,Giannetti2015,Wang2023}.

The purpose of this article is to investigate the $W^{1,2}_\loc$-regularity of the vector field
\begin{equation} \label{eq:Intro-V-psi}
    V_\psi(\nabla u):=\frac{\psi'(|\nabla u|)}{|\nabla u|}\nabla u,
\end{equation}
where $u$ denotes a weak solution of \eqref{eq:OL-equation} and $\psi$ denotes a $C^2$-regular Orlicz-function with $(\tilde{p},\tilde{q})$ growth for $1<\tilde{p}\leq\tilde{q}<\infty$. We show that if the growth rates of $\vp$ and $\psi$ are close enough in a suitable quantitative sense, then it is possible to derive a Caccioppoli--type inequality for $V_\psi(\nabla u)$. The proof relies on Cordes matrix inequalities \cite{Cordes1961} (see also \cite{Talenti1965}) and a pointwise inequality from \cite{Haarala2022}.

Introduction of the new Orlicz function $\psi$ allows us to cover multiple second order regularity results at once. Some particularly interesting special cases of $\psi$ are given by $\psi(t)=\vp(t)$, $\psi'(t)=\sqrt{\vp'(t)t}$, and $\psi(t)=\frac{1}{2}t^2$. In these cases $V_\psi(\nabla u)$ is given by
\begin{equation} \label{eq:interesting-vector-fields}
    \frac{\vp'(\nabla u)}{|\nabla u|}\nabla u,\quad
    \sqrt{\frac{\vp'(\nabla u)}{|\nabla u|}}\nabla u\quad\text{and}\quad \nabla u, 
\end{equation}
respectively. If $\vp(t)=\frac{1}{p}t^p$, then the first vector field in \eqref{eq:interesting-vector-fields} is $|\nabla u|^{p-2}\nabla u$ and the second is $|\nabla u|^{\frac{p-2}{2}}\nabla u$. In Section \ref{ssec:Discussion} we discuss the meaning of our results for the special cases of vector fields from \eqref{eq:interesting-vector-fields}.

For the purposes of this article, we introduce a \emph{closeness function} $\theta\colon(0,\infty)\to(0,\infty)$ of two Orlicz functions $\vp$ and $\psi$. It is defined by
\begin{equation} \label{eq:Intro-closeness-def}
    \theta(t):=
    \Big(\frac{\vp''(t)t}{\vp'(t)}\Big)\Big/\Big(\frac{\psi''(t)t}{\psi'(t)}\Big).
\end{equation}
Roughly speaking, $\theta$ is the ratio of the pointwise growth rates of $\vp'$ and $\psi'$. For a detailed discussion on the closeness function, see Section \ref{ssec:closenss}.

We consider the homogeneous case $f=0$ and the non-homogeneous case $f\in L^\infty(\Om)$ separately. 

For the homogeneous problem, our main result, Theorem \ref{thm:main1} below, is a generalization of \cite{Dong2020}*{Theorem 1.1} and \cite{Sarsa2022}*{Theorem 1.1}. These known results concern the $p$-Laplace equation $\Delta_pu=0$, and they yield $W^{1,2}_{\loc}$-regularity of $|\nabla u|^\beta\nabla u$ for $\beta>-1+\frac{(n-2)(p-1)}{2(n-1)}$. Our theorem says that this result can be straightforwardly generalized to the case of Orlicz--Laplace equation, and the appropriate condition is given in terms of the closeness function $\theta$.

In the statement of the theorem, and throughout the paper, a generic ball in $\Rn$ with radius $r>0$ is denoted briefly as $B_r$. The integral average of a locally integrable function $v$ over $B_r$ is denoted by
$$ (v)_{B_r}:=\fint_{B_r}vdx=\frac{1}{|B_r|}\int_{B_r}vdx. $$

\begin{theorem}[Main result for homogeneous problem] \label{thm:main1}
Let $\vp$ be a $C^2$-regular Orlicz function with $(p,q)$ growth and let $u\in W^{1,\vp}(\Om)$ be a weak solution of \eqref{eq:OL-equation} with $f=0$. 
Suppose that $\psi$ is another $C^2$-regular Orlicz function with $(\tilde{p},\tilde{q})$ growth and $V_\psi(\nabla u):=\frac{\psi'(|\nabla u|)}{|\nabla u|}\nabla u$.
If the closeness function $\theta$ of $\vp$ and $\psi$, as defined in \eqref{eq:Intro-closeness-def}, is uniformly bounded from above as follows,
\begin{equation} \label{eq:Intro-closeness}
    \sup\{\theta(t):0<t<\infty\}<\frac{2(n-1)}{n-2},
\end{equation}
then $V_\psi(\nabla u)\in W^{1,2}_{\loc}(\Om;\Rn)$. Moreover, we have the local estimate
\begin{equation} \label{eq:main1}
    \int_{B_r}|D(V_\psi(\nabla u))|^2dx
\leq \frac{C}{r^2}\int_{B_{2r}}|V_\psi(\nabla u)-(V_\psi(\nabla u))_{B_{2r}}|^2dx,
\end{equation}
for all concentric balls $B_r\subset B_{2r}\Subset\Om$.
Here $C>0$ is a positive constant that depends on $n$, $p$, $q$, $\tilde{p}$, $\tilde{q}$, and the supremum in \eqref{eq:Intro-closeness}.
\end{theorem}

Note that the right hand side of \eqref{eq:main1} is finite, because $\nabla u$ is locally bounded \cites{Lieberman1991,Lieberman1992}. 
Combining the Caccioppoli type inequality \eqref{eq:main1} with a Sobolev--Poincar\'e inequality and utilizing Gehring's lemma, we obtain a higher integrability result.

\begin{corollary} [Higher integrability for homogeneous problem] \label{thm_higher-integrability-1}
Under the assumptions of Theorem \ref{thm:main1}, there exists a positive constant $\delta>0$ such that $D(V_\psi(\nabla u))\in L^{2+\delta}_{\loc}(\Omega;\R^{n\times n})$. Additionally, 
\begin{equation*}
    \left(\fint_{B_r} |D(V_\psi(\nabla u))|^{2+\delta} \, dx\right)^{1/(2+\delta)} 
    \leq C \left(\fint_{B_{2r}} |D(V_\psi(\nabla u))|^2 \, dx\right)^{1/2}
\end{equation*}
for all concentric balls $B_r\subset B_{2r}\Subset\Om$, where $C>0$ is a positive constant.
The constants $\delta$ and $C$ depend on $n$, $p$, $q$, $\tilde{p}$, $\tilde{q}$, and the supremum in \eqref{eq:Intro-closeness}.
\end{corollary}

For the non-homogeneous problem, let us recall that a natural approach is to consider regularity of the vector field $V_\vp(\nabla u)$, that is, to choose $\psi=\vp$. In this paper our aim is to investigate the possibility that $\psi\neq\vp$.

Our work is closely related to the work of Challal and Lyaghfouri \cite{Challal2010}. They considered the vector field $\sqrt{\frac{\vp'(|\nabla u|)}{|\nabla u|}}\nabla u$ and obtained the following result: If $u\in W^{1,\vp}(\Om)$ is a weak solution of \eqref{eq:OL-equation} and either $\nabla f$ belongs to the dual Orlicz-space or $t\mapsto \frac{\vp'(t)}{t}$ is non-increasing, then $D\Big(\sqrt{\frac{\vp'(|\nabla u|)}{|\nabla u|}}\nabla u\Big)\in L^2_{\loc}(\Om;\R^{n\times n})$. Our results improve and generalize \cite{Challal2010}*{Theorem 3.1} in the latter case. Firstly, instead of the special vector field $\sqrt{\frac{\vp'(|\nabla u|)}{|\nabla u|}}\nabla u$, we consider a more general vector field $\frac{\psi'(|\nabla u|)}{|\nabla u|}\nabla u$. Secondly, we obtain a Caccioppoli-type inequality \eqref{eq:main2-new}, which in turn allows us to get a slightly better integrability for $D(\frac{\psi'(|\nabla u|)}{|\nabla u|}\nabla u)$. For a more detailed discussion, see Example \ref{ex:CL-improvement}.

In this article we derive a Caccioppoli-type inequality, inequality \eqref{eq:main2-new} below, for solutions of the non-homogeneous problem. The central assumption is the closeness of the Orlicz functions $\vp$ and $\psi$, as in Theorem \ref{thm:main1}. Additionally, the source term on the right hand side of the Caccioppoli-type inequality \eqref{eq:main2-new} involves the ratio $\psi'(|\nabla u|)/\vp'(|\nabla u|)$. If $\psi=\vp$, then such ratio reduces to constant one, but in general we need to impose a further condition to the Orlicz function $\psi$ in order to guarantee the finiteness of such integral. We introduce the \emph{ratio function} of $\vp$ and $\psi$, given by
\begin{equation} \label{eq:Intro-ratio-def}
    \rho(t):=\frac{\psi'(t)}{\vp'(t)},
\end{equation}
and we assume that $\rho(t)$ stays uniformly bounded for $t\in[0,1]$.

\begin{theorem}[Main result for non-homogeneous problem] \label{thm:main2-new}
Let $\vp$ be a $C^2$-regular Orlicz function with $(p,q)$ growth and let $u\in W^{1,\vp}(\Om)$ be a weak solution of \eqref{eq:OL-equation} with $f\in L^\infty(\Om)$. Suppose that $\psi$ is another $C^2$-regular Orlicz function with $(\tilde{p},\tilde{q})$ growth and $V_\psi(\nabla u):=\frac{\psi'(|\nabla u|)}{|\nabla u|}\nabla u$. If the closeness function $\theta$ of $\vp$ and $\psi$, as defined in \eqref{eq:Intro-closeness-def}, is uniformly bounded from above as follows, 
\begin{equation} \label{eq:Intro-closeness-2}
    \sup\{\theta(t):0<t<\infty\}<\frac{2(n-1)}{n-2},
\end{equation}
and the ratio function $\rho$ of $\vp$ and $\psi$, as defined in \eqref{eq:Intro-ratio-def}, is uniformly bounded for small values as follows,
\begin{equation} \label{eq:Intro-ratio-bounded}
    \sup\{\rho(t):0\leq t\leq 1\}<\infty,
\end{equation}
then $V_\psi(\nabla u)\in W^{1,2}_{\loc}(\Om;\Rn)$. Moreover, we have the local estimate
\begin{align} \label{eq:main2-new}
    \int_{B_r}|D(V_\psi(\nabla u))|^2dx
    \leq \frac{C}{r^2}\int_{B_{2r}}|V_\psi(\nabla u)-(V_\psi(\nabla u))_{B_{2r}}|^2dx
    +C\int_{B_{2r}}\Big(\frac{\psi'(|\nabla u|)}{\vp'(|\nabla u|)}f\Big)^2dx,
\end{align}
for all concentric balls $B_r\subset B_{2r}\Subset\Om$. Here $C>0$ is a positive constant that depends on $n$, $p$, $q$, $\tilde{p}$, $\tilde{q}$, and the supremum in \eqref{eq:Intro-closeness-2}.
\end{theorem}

The right hand side of \eqref{eq:main2-new} is finite due to the boundedness of $\nabla u$ and the assumption \eqref{eq:Intro-ratio-bounded}. 
Similarly as in the homogeneous case, the Caccioppoli--type inequality \eqref{eq:main2-new} yields a higher integrability result.

\begin{corollary} [Higher integrability for non-homogeneous problem] \label{thm_higher-integrability-2-new}
Under the assumptions of Theorem \ref{thm:main2-new}, there exists a positive constant $\delta>0$
such that $D(V_\psi(\nabla u))\in L^{2+\delta}_{\loc}(\Omega;\R^{n\times n})$. Additionally, 
\begin{align*}
    \left(\fint_{B_r} |D(V_\psi(\nabla u))|^{2+\delta} \, dx\right)^{1/(2+\delta)} 
    &\leq C 
    \Big(\left(\fint_{B_{2r}} |D(V_\psi(\nabla u))|^2 \, dx\right)^{1/2} \\
    &\quad
    +\Big(\fint_{B_{2r}}\Big(\frac{\psi'(|\nabla u|)}{\vp'(|\nabla u|)}f\Big)^{2+\delta}dx\Big)^{1/(2+\delta)}\Big)
\end{align*}
for all concentric balls $B_r\subset B_{2r}\Subset\Om$, where $C>0$ is a positive constant. The constants $\delta$ and $C$ depend on $n$, $p$, $q$, $\tilde{p}$, $\tilde{q}$, and the supremum in \eqref{eq:Intro-closeness-2}.
\end{corollary}

Theorem \ref{thm:main2-new} and Corollary \ref{thm_higher-integrability-2-new} are interesting even in the $p$-Laplacian case $\Delta_pu=f$, as the non-homogeneous equation was not studied in the earlier work by Dong, Peng, Zhang and Zhou \cite{Dong2020} and the second author \cite{Sarsa2022}. Montoro, Muglia and Sciunzi show in their recent preprint \cite{Montoro2023} that $|\nabla u|^{\alpha-1}\nabla u\in W^{1,2}(\Om;\Rn)$ under the assumptions that $f\in W^{1,1}(\Om)\cap L^q(\Om)$ and $\alpha>\frac{p-1}{2}$. We compare our results to theirs in Section \ref{ssec:Discussion-optimality}.

Let us elaborate on the essence of our proofs. The main technical tool is a pointwise estimate, Lemma~\ref{lem:Inequality}, which can be found in \cite{Haarala2022}*{Lemma 3.2, Corollary 3.3}. It is based on Cordes' matrix inequalities \cite{Cordes1961}, see also \cite{Talenti1965}. Similar inequalities appear also in \cite{Cianchi2018}. 

We would like to use the pointwise estimate in the following form
\begin{equation} \label{eq:Intro-Ineq}
\begin{aligned}
    c|D(V_\psi(\nabla u))|^2
    &\leq
    \diverg((D(V_\psi(\nabla u))-\tr(D(V_\psi(\nabla u)))I)(V_\psi(\nabla u)-Z)) \\
    &\quad
    +C\Big(\frac{\psi'(|\nabla u|)}{\vp'(|\nabla u|)}\Big)^2\big(\diverg(V_\vp(\nabla u))\big)^2.
\end{aligned}
\end{equation}
Here $\tr$ refers to trace of a matrix, $I$ denotes the identity matrix and $Z\in\Rn$ is an arbitrary vector. This would require $C^3$-regularity for solution $u$, but  weak solutions of \eqref{eq:OL-equation} only have $C^{1,\alpha}$-regularity in general, see Lieberman \cites{Lieberman1991,Lieberman1992}. Hence a regularization of \eqref{eq:OL-equation} is necessary. 

We employ a similar regularization as in the work of Challal and Lyaghfouri \cite{Challal2010}, namely for $\epsilon>0$ we consider
$$ \diverg\Big(\frac{\vp'(\sqrtue)}{\sqrtue}\nabla\ue\Big)=f. $$
Weak solutions $\ue$ of the above equation belong to $W^{2,2}_{\loc}$, as shown by \cite{Challal2010}*{Theorem 2.1}. Nonetheless, we still cannot write the estimate \eqref{eq:Intro-Ineq} with $u$ replaced with $\ue$, due to the possible lack of regularity of the Orlicz functions $\vp$ and $\psi$, as well as the data $f$. To circumvent this problem, we prove a weak type version of inequality \eqref{eq:Intro-Ineq}, see Lemma \ref{lem:main-lemma-delta0-kappa0}. This is obtained by making further regularizations to both our Orlicz-functions $\vp$ and $\psi$ and weak solutions $\ue$. A careful analysis of convergences of regularization parameters and their dependencies on relevant constants allows us to generalize the pointwise inequality \eqref{eq:Intro-Ineq} to an integral form suitable for our main result.

This article is organized as follows. The second section contains all the preliminaries such as Orlicz functions, Orlicz spaces and the notation used in the article. It also contains the definition of the regularized equation and introduces the pointwise equality used as a main technical tool. The third section is devoted to the discussion of the assumptions of our main results. We discuss both the closeness function $\theta$ and the ratio $t\mapsto \psi'(t)/\vp'(t)$ that appears in the estimate \eqref{eq:main2-new} in the non-homogeneous case. We compute these functions in certain special cases, and thus we demonstrate various special corollaries of our main results. We also discuss the optimality question of our results. In the fourth section we work with the regularization schemes required for the use of the pointwise estimate. We prove a version of Theorem \ref{thm:main2-new} with $\ue$. The fifth section contains the passage to the limit when $\epsilon\to0$ to conclude the proofs of our main results.
%--------------------------------------------------------------------
\section{Preliminaries} \label{sec:prel}
%--------------------------------------------------------------------
\subsection{Orlicz functions, Orlicz spaces and Orlicz--Laplace equation}\label{ssec:orlicz-spaces}

In this section we state the precise definitions of Orlicz functions, Orlicz spaces, and weak solutions to the Orlicz--Laplace equation \eqref{eq:OL-equation}. We also recall some of their basic properties. For details, we refer to the monographs \cites{Bennett1988,Krasnoselskii1961,Rao1991,Adams2003}.

We say $\vp:[0,\infty) \to [0,\infty)$ is an Orlicz function if it is convex and satisfies $\lim_{t \to 0^+} \vp(t) = \vp(0)=0$ and $\lim_{t \to \infty} \vp(t) = \infty$. In this paper we consider $C^2$-regular Orlicz functions that satisfy $(p,q)$ growth conditions. More precisely, we assume $\vp\in C^1([0,\infty)] \cap C^2 ((0,\infty))$ and
\begin{equation} \label{eq:growth-rate}
    0<p-1\leq\dfrac{\vp''(t)t}{\vp'(t)} \leq q-1 < \infty.
\end{equation}
These assumptions imply the commonly known $\Delta_2$ and $\nabla_2$ conditions. Especially our Orlicz functions are doubling as is often assumed when dealing with regularity results. More precisely, $\vp(2t) \leq C \vp(t)$, where the constant $C$ depends only on $q$ and $\vp(1)$. The doubling condition excludes for instance exponential growth of the Orlicz function.

Common examples of Orlicz functions covered in this article include
\begin{itemize}
    \item $\vp(t) = \frac{1}{p}t^p$ for $p>1$;
    \item $\vp(t) = t^p \log^\alpha(C+t)$ for $p>1$, $\alpha\in\R$ and a (sufficiently large) constant $C>0$;
    \item $\vp(t) = t^p + at^q$ for $p,q>1$ and $a>0$;
    \item $C^2$-regular version of $\vp(t) = \max\{t^p,t^q\}$ for $p,q>1$;
    \item products and compositions of previous examples.
\end{itemize}

For the definition of Orlicz spaces, let $\Omega \subset \Rn$, $n\geq 2$, denote a bounded domain and $\vp$ an Orlicz function, as described above. We define the Orlicz space $L^\vp(\Om)$ as
\begin{equation*}
    L^\vp(\Omega) := \left \{u \in L^1(\Omega) : \lim_{\lambda \to 0^+} \int_\Omega \vp(\lambda |u|) \, dx = 0 \right\}.
\end{equation*}
This space equipped with the Luxemburg norm
\begin{equation*}
    \|u\|_{L^\vp(\Omega)} := \inf \left \{\lambda >0 : \int_\Omega \vp\left( \dfrac{|u|}{\lambda}\right) \, dx \leq 1 \right\}
\end{equation*}
is a Banach space. The Orlicz function $\vp(t)=\frac{1}{p}t^p$ generates the standard $L^p(\Omega)$-spaces, and $\vp(t)=t^p\log^\alpha(1+t)$ generates $L^p \log^\alpha L$-spaces, also known as Zygmund spaces.

We also need the Orlicz--Sobolev space $W^{1,\vp}(\Omega)$ which consists of functions $u\in L^\vp(\Omega)$ that posses weak derivatives belonging to $L^\vp(\Omega)$. The norm of this space is defined as
\begin{equation*}
    \|u\|_{W^{1,\vp}(\Omega)} := \|u\|_{L^\vp(\Omega)} + \||\nabla u|\|_{L^\vp (\Omega)},
\end{equation*}
which makes $W^{1,\vp}(\Om)$ into a reflexive Banach space. With a slight abuse of notation we abbreviate $\||\nabla u|\|_{L^\vp(\Omega)} $ as $\|\nabla u\|_{L^\vp(\Omega)}$. 

We say $u\in W^{1,\vp}(\Om)$ is a weak solution of Orlicz-Laplace equation \eqref{eq:OL-equation} if
\begin{equation*}
    \int_{\Om}\frac{\vp'(|\nabla u|)}{|\nabla u|}\la \nabla u,\nabla \eta\ra dx=\int_\Om f\eta dx
\end{equation*}
holds for all smooth compactly supported test functions $\eta\in\CcO$. It is well-known that such weak solutions $u$ belong to $C^{1,\alpha}_{\loc}(\Om)$ with $\alpha=\alpha(n,p,q)\in(0,1)$ and with the following estimate: for any subdomain $U\Subset\Om$
\begin{equation}\label{eq:u-is-C1}
    \|u\|_{C^{1,\alpha}(U)}\leq 
    C=C(n,p,q,\vp'(1),\|f\|_{L^\infty(\Om)},\|u\|_{L^\infty(\Om)},\dist(U,\partial\Om)).  
\end{equation}
See Lieberman \cites{Lieberman1991,Lieberman1992}. See also Hästö and Ok \cite{Hasto2022} for $f=0$.
%--------------------------------------------------------------------
\subsection{Approximate solution $\ue$ and its known regularity}

Let $u\in W^{1,\vp}(\Om)$ be a weak solution of \eqref{eq:OL-equation}. We approximate $u$ with solution $\ue$ of a regularized problem, similarly to \cite{Challal2010}. Namely, for a smooth subdomain $U\Subset \Om$ and $0<\epsilon\leq 1$, consider the Dirichlet problem
\begin{equation} \label{eq:Regularized-Dirichlet-Problem}
\begin{cases}
\begin{aligned}
-\diverg\Big(\frac{\vp'(\sqrtue)}{\sqrtue}\nabla\ue\Big)=f &\quad\text{in }U;\\
\ue=u &\quad\text{on }\partial U.
\end{aligned}
\end{cases}
\end{equation}
The problem \eqref{eq:Regularized-Dirichlet-Problem} has a unique solution in $W^{1,\vp}(U)$, see for instance \cite{Giusti2003}. Let us consider the regularity of the weak solution of the approximating problem \eqref{eq:Regularized-Dirichlet-Problem}. For the proof of the following results, we refer to \cite{Challal2010}.

We have that $\ue\in C^{1,\alpha}_{\loc}(U)$, with $\alpha=\alpha(n,p,q)\in(0,1)$, and by \cite{Challal2010}*{Theorem 2.1} $\ue\in W^{2,2}_{\loc}(U)$. In addition $\{\ue\}_\epsilon$ is uniformly bounded in $C^{1,\alpha}_{\loc}(U)$, namely, for any subdomain $U'\Subset U$
\begin{equation} \label{eq:Uniform-bound-in-C1alpha}
    \|\ue\|_{C^{1,\alpha}(U')}\leq 
    C=C(n,p,q,\vp'(1),\|f\|_{L^\infty(U)},\|u\|_{L^\infty(U)},\dist(U',\partial U)).
\end{equation}
By \cite{Challal2010}*{Lemma 3.3} we know that there exists a subsequence of $\{\ue\}_\epsilon$, still denoted by itself, such that 
\begin{equation} \label{eq:Convergence-of-ue-in-C1}
    \ue\xrightarrow{\epsilon\to0}u\quad\text{in }C^1_{\loc}(U).
\end{equation}
%--------------------------------------------------------------------
\subsection{Pointwise inequality} \label{ssec:inequality}
In this section we state the key tool of this article, a pointwise differential inequality for nonlinear gradient fields, Lemma \ref{lem:Inequality} below. Such inequality was studied in detail by Haarala and the second author in \cite{Haarala2022}*{Sections 2 and 3}. The proof relies on the work of Cordes \cite{Cordes1961} and Talenti \cite{Talenti1965}.

For the statement of the lemma, we introduce notation that is similar to \cite{Haarala2022}. This notation serves only for the lemma and its application Section \ref{ssec:application}.

Let $a,b\colon[0,\infty)\to(0,\infty)$ be $C^2$-functions and denote
\begin{align}
    \alpha(t):=\frac{a'(t)t}{a(t)}
    \quad\text{and}\quad
    \beta(t):=\frac{b'(t)t}{b(t)}
    \quad\text{and}\quad
    \gamma(t):=\frac{\alpha(t)+1}{\beta(t)+1}.
\end{align}
We denote the infima and suprema of these quantities over an interval $[0,C_1]$, where $C_1>0$, as follows:
\begin{align*}
    i_\alpha:=\inf\{\alpha(t): 0\leq t\leq C_1\} \quad\text{and}\quad
    s_\alpha:=\sup\{\alpha(t): 0\leq t\leq C_1\},
\end{align*}
and
$$ i_\beta:=\inf\{\beta(t): 0\leq t\leq C_1\} \quad\text{and}\quad s_\beta:=\sup\{\beta(t): 0\leq t\leq C_1\}, $$
and
$$ s_\gamma:=\sup\{\gamma(t): 0\leq t\leq C_1\}. $$
The lemma below holds also for $C_1=\infty$, but for the later use, we need to restrict ourselves to a bounded interval $[0,C_1]$.
    
\begin{lemma} \label{lem:Inequality}
Suppose that $a,b\colon[0,\infty)\to(0,\infty)$ are $C^2$-functions such that
\begin{align} \label{eq:ellipticity-for-ineq}
    -1<i_\alpha\leq s_\alpha<\infty
    \quad\text{and}\quad
    -1<i_\beta\leq s_\beta<\infty.
\end{align}
Suppose moreover that
\begin{equation} \label{eq:Cordes-for-AwrtB}
    s_\gamma<\frac{2(n-1)}{n-2}.
\end{equation}
Let $u\in C^3(\Omega)$ such that $\|\nabla u\|_{L^\infty(\Om)}\leq C_1$ and suppose that $V_a:=a(|\nabla   u|)\nabla   u\in C^2(\Om)$ and \\
$V_b:=b(|\nabla   u|)\nabla   u\in C^2(\Om)$. Then have the inequality
\begin{equation} \label{eq:Inequality}
    c|D  V_b|^2\leq
    \diverg((D  V_b-\tr(DV_b)I)(V_b-Z))
    +C\Big(\frac{b(|\nabla u|)}{a(|\nabla u|)}\Big)^2\big(\diverg(V_a)\big)^2,
\end{equation}
where $Z\in\Rn$ is an arbitrary vector and
$c=c(i_\alpha,s_\alpha,i_\beta,s_\beta,s_\gamma,n)>0$ and \\
$C=C(i_\alpha,s_\alpha,i_\beta,s_\beta,s_\gamma,n)>0$ are positive constants.
\end{lemma}

\begin{remark}
In the original statement of the above Lemma in \cite{Haarala2022}, we have the additional condition
$$ i_\gamma:=\inf\{\gamma(t):0\leq t\leq C_1\}>0. $$
However, this follows trivially from \eqref{eq:ellipticity-for-ineq}.
\end{remark}

\begin{remark}
Formally, we want to employ the above lemma with
$$ a(t)=\frac{\vp'(t)}{t}\quad\text{and}\quad b(t)=\frac{\psi'(t)}{t}. $$
Proceeding with a formal computation,
$$ \alpha(t)=\frac{a'(t)t}{a(t)} = \nu(t)-1 
\quad\text{and}\quad
\beta(t)=\frac{b'(t)t}{b(t)} = \mu(t)-1
$$
and moreover,
$$ \gamma(t)=\frac{\alpha(t)+1}{\beta(t)+1}
=\frac{\nu(t)}{\mu(t)}=\theta(t).
$$
The quantities $\nu$ and $\mu$ are introduced Section \ref{ssec:closenss} below. See Section \ref{ssec:application} for a rigorous proof built upon the above computation.
\end{remark}
%--------------------------------------------------------------------
\section{Orlicz function $\psi$ close to $\vp$} \label{sec:intermediate}
%--------------------------------------------------------------------
\subsection{Closeness function and ratio function of two Orlicz functions} \label{ssec:closenss}

Let $\vp$ be an Orlicz function, as described in Section \ref{sec:prel}. For our investigation we need another Orlicz function, that we denote by $\psi$ throughout the article. We assume that $\psi\in C^1([0,\infty))\cap C^2((0,\infty))$ and that $\psi$ satisfies $(\tilde p, \tilde q)$ growth conditions for some $1<\tilde{p}\leq\tilde{q}<\infty$. 

Given the Orlicz functions $\vp$ and $\psi$, we denote
\begin{equation} \label{eq:nu-and-mu}
    \nu(t):=\frac{\vp''(t)t}{\vp'(t)} 
    \quad\text{and}\quad 
    \mu(t):=\frac{\psi''(t)t}{\psi'(t)}.
\end{equation}
We define the closeness function of $\vp$ and $\psi$ by
\begin{equation} \label{eq:theta}
    \theta(t) := \frac{\nu(t)}{\mu(t)}.
\end{equation}
The role of the closeness function is to measure the pointwise difference of growth rates of $\vp$ and $\psi$. It is a continuous and bounded function on $(0,\infty)$ and
\begin{equation} \label{eq:trivial-bounds-for-theta}
    \frac{p-1}{\tilde{q}-1}\leq \theta(t)\leq \frac{q-1}{\tilde{p}-1}
\end{equation}
due to the $(p,q)$ growth of $\vp$ and $(\tilde{p},\tilde{q})$ growth of $\psi$.

For our main results, we assume a potentially stronger upper bound for the closeness function $\theta$. The closeness condition \eqref{eq:Intro-closeness} in Theorem \ref{thm:main1} stands as
\begin{equation} \label{eq:closenss-for-discussion}
    s_\theta:=\sup\{\theta(t):t>0\} <\frac{2(n-1)}{n-2}.
\end{equation}
Notice that the bounds \eqref{eq:trivial-bounds-for-theta} imply that whenever
$$ \frac{q-1}{\tilde{p}-1}<\frac{2(n-1)}{n-2}, $$
then the closeness condition \eqref{eq:closenss-for-discussion} follows trivially from the controlled growth rates of $\vp$ and $\psi$. In particular this is always the case when $n=2$.

For Theorem \ref{thm:main2-new}, which concerns the non-homogeneous equation, we need the ratio function of $\vp$ and $\psi$, given by
$$ \rho(t):=\frac{\psi'(t)}{\vp'(t)}. $$
Due to the growth conditions of $\vp$ and $\psi$, the ratio function $\rho$ satisfies
$$ \rho(t)\simeq \frac{\psi(t)}{\vp(t)}. $$
The multiplicative constants depend on $p$, $q$, $\tilde{p}$ and $\tilde{q}$. For Theorem \ref{thm:main2-new} we need that
\begin{align} \label{eq:ratio-bounded}
    s_\rho:=\sup\{\rho(t):0\leq t\leq 1\}<\infty.
\end{align}
Roughly speaking, this means that $\psi(t)$ tends to zero faster than $\vp(t)$, when $t\to0$.
%--------------------------------------------------------------------
\subsection{Discussion on special choices of $\psi$} \label{ssec:Discussion}

In this section we present some examples of the closensss function $\theta$ and the ratio function $\rho$ for various Orlicz functions $\vp$ and $\psi$. We discuss what kind of regularity results they imply in light of Theorems \ref{thm:main1} and \ref{thm:main2-new}, and Corollaries \ref{thm_higher-integrability-1} and \ref{thm_higher-integrability-2-new}.

\begin{example}[Power functions] \label{ex:powers}
The basic example is given by choosing $\vp$ and $\psi$ to be power functions. Set $\vp(t)=\frac{1}{p}t^p$ and $\psi(t)=\frac{1}{\beta+2}t^{\beta+2}$ for $p>1$ and $\beta>-1$. Then the Orlicz--Laplace equation \eqref{eq:OL-equation} reduces to the usual $p$-Laplace equation. In this case
$$ \theta(t)\equiv\frac{p-1}{\beta+1} \quad\text{and}\quad 
\rho(t)=t^{\beta+1-(p-1)} $$
The closeness condition \eqref{eq:closenss-for-discussion} and the ratio condition \eqref{eq:ratio-bounded} hold whenever
$$ \beta+1>\frac{(n-2)(p-1)}{2(n-1)}\quad\text{and}\quad
\beta+1\geq p-1,
$$
respectively. Consequently, $|\nabla u|^\beta\nabla u\in W^{1,2+\delta}_{\loc}(\Om)$ for weak solutions of $\Delta_pu=0$, holds whenever \\
$\beta>-1+\frac{(n-2)(p-1)}{2(n-1)}$. This result was obtained in \cites{Dong2020,Sarsa2022}. See also \cite{Bojarski1987}*{Section 2} for the planar case $n=2$. Moreover, $|\nabla u|^\beta\nabla u\in W^{1,2+\delta}_{\loc}(\Om;\Rn)$ for weak solutions of $\Delta_pu=f$, holds whenever $\beta\geq p-2$.
\end{example}

\begin{example}[Natural quantity for non-homogeneous equation] \label{ex:natural-non-homog}
If $\psi=\vp$, then
$$ \theta(t)\equiv1 \quad\text{and}\quad 
\rho(t)\equiv 1. $$
Both the closeness condition \eqref{eq:closenss-for-discussion} and the ratio condition \eqref{eq:ratio-bounded} hold trivially.

Consequently $\frac{\vp'(|\nabla u|)}{|\nabla u|}\nabla u \in W^{1,2+\delta}_{\loc}(\Omega;\Rn)$ for weak solutions of the Orlicz--Laplace equation \eqref{eq:OL-equation} whenever $\vp$ satisfies $\Delta_2$ and $\nabla_2$ conditions. For $\vp(t) = \frac{1}{p}t^p$, this recovers the result ${|\nabla u|^{p-2}\nabla u \in W^{1,2+\delta}_{\loc}(\Omega;\Rn)}$. This type of regularity was studied by Cianchi and Maz'ya \cite{Cianchi2018} under the assumption that $L^2(\Om)$.
\end{example}

\begin{example}[$W^{2,2}_{\loc}$-regularity] 
If $\psi(t)=\frac{1}{2}t^2$, then
$$ \theta(t)=\nu(t) \quad\text{and}\quad 
\rho(t)=\frac{t}{\vp'(t)}. $$
The closeness condition \eqref{eq:closenss-for-discussion} and the ratio condition \eqref{eq:ratio-bounded} hold whenever
$$ q<3+\frac{2}{n-2}\quad\text{and}\quad q\leq 2, $$
respectively. In this cases we obtain that $u\in W^{2,2+\delta}_{\loc}(\Om)$. This is in line with the range $1<p<3+\frac{2}{n-2}$ for solutions of $\Delta_p u=0$, due to Manfredi and Weitsman \cite{Manfredi1988}.
\end{example}

\begin{example}[Natural quantity for homogeneous equation] \label{ex:CL-improvement}
If $\psi'(t)=\sqrt{\varphi'(t)t}$, then 
$$ \theta(t)=\frac{2\vp''(t)t}{\vp''(t)t+\vp'(t)}
\quad\text{and}\quad 
\rho(t)=\sqrt{\frac{t}{\vp'(t)}}. $$
The closeness condition \eqref{eq:closenss-for-discussion} is always satisfied, and the ratio condition \ref{eq:ratio-bounded} holds whenever $q\leq2$.

Consequently, $\sqrt{\frac{\varphi'(|\nabla u|)}{|\nabla u|}} \nabla u \in W^{1,2+\delta}_{\loc}(\Omega;\Rn)$ for weak solutions of the homogeneous Orlicz--Laplace equation. This regularity result is known ($\delta=0$), see for instance \cite{Diening2009}*{Corollary 3.7}. In the case of $\vp(t) = \frac{1}{p}t^p$  we recover the well-known result that $|\nabla u|^{\frac{p-2}{2}}\nabla u \in W^{1,2+\delta}_{\loc}(\Omega;\Rn)$ for all $1<p<\infty$. 

For the non-homogeneous Orlicz--Laplace equation we obtain that  $\sqrt{\frac{\varphi'(|\nabla u|)}{|\nabla u|}} \nabla u \in W^{1,2+\delta}_{\loc}(\Omega;\Rn)$ whenever $q\leq 2$. This type of regularity (with $\delta=0$) was obtained earlier by Challal and Lyaghfouri \cite{Challal2010}.
\end{example}
%--------------------------------------------------------------------
\subsection{Discussion on optimality} \label{ssec:Discussion-optimality}
We discuss the optimality of Theorems \ref{thm:main1} and \ref{thm:main2-new} in the simple case of power functions, when $\vp(t)=\frac{1}{p}t^p$ and $\psi(t)=\frac{1}{\beta+2}t^{\beta+2}$.

In this case Theorem \ref{thm:main1} yields a result for $p$-harmonic functions. The closeness condition \eqref{eq:Intro-closeness} holds if
$$ \beta>-1+\frac{(n-2)(p-1)}{2(n-1)}. $$
In the planar case $n=2$ this reduces to 
$$ \beta>-1. $$
This is clearly the best range one can hope for in our set up, since for $\beta=-1$ the Orlicz function $\psi$ fails to satisfy the desired growth conditions.

Besides, the simple example of the planar harmonic function $w(x)=x_1^2-x_2^2$ shows that in general \\
$|\nabla u|^{-1}\nabla u\notin W^{1,2}_{\loc}(\R^2;\R^2)$. Indeed, $|\nabla w|^{-1}\nabla w=|x|^{-1}x$, and
\begin{align*}
    D(|x|^{-1}x)
    %&=|x|^{-1}I+x\otimes (-|x|^{-3}x) \\
    &=|x|^{-1}(I-|x|^{-2}x\otimes x)
\end{align*}
and
$$ |D(|x|^{-1}x)|^2=2|x|^{-2}\notin L^1_{\loc}(\R^2) $$
For a proof that the range $\beta>-1$ is optimal when $p\neq 2$, we refer to \cite{Dong2022}*{Appendix A} and \cite{Iwaniec1989}.

To the best of our knowledge, it is not known if the range $\beta>-1+\frac{(n-2)(p-1)}{2(n-1)}$ is optimal in higher dimensions $n\geq3$. Nonetheless, the method of this article, which relies on Cordes' matrix inequalities, is not capable to reach better range.

Theorem \ref{thm:main2-new} yields the $W^{1,2}_{\loc}$-regularity of $|\nabla u|^\beta\nabla u$ for solutions of $\Delta_p u=f$, where $f\in L^\infty$. As demonstrated in Example \ref{ex:powers}, in this case the ratio condition dominates the closeness condition, and Theorem \ref{thm:main2-new} is applicable on the range
$$ \beta\geq p-2. $$
This is indeed a natural assumption in order to guarantee that the last integral in the estimate \eqref{eq:main2-new} of Theorem \ref{thm:main2-new},
$$ \int \big(|\nabla u|^{\beta-(p-2)}f\big)^2dx, $$
is finite. The endpoint $\beta=p-2$ yields Sobolev regularity for the stress field $|\nabla u|^{p-2}\nabla u$.

However, the range $\beta\geq p-2$ might not be optimal. For example, Montoro, Muglia and Sciunzi \cite{Montoro2023} cover a wider range of $\beta$, namely they prove that $|\nabla u|^\beta\nabla u\in W^{1,2}(\Om)$ if $\beta>-1+\frac{p-1}{2}$, under the assumption that $f\in W^{1,1}(\Om)\cap L^q(\Om)$. By \cite{Montoro2023}*{Remark 1.4}, the range $\beta>-1+\frac{p-1}{2}$ is the best one can hope for, since 
$$ v(x)=\tfrac{p-1}{p}|x_1|^{\frac{p}{p-1}} $$
solves $\Delta_p v=1$ and $D(|\nabla v|^\beta\nabla v)\in L^2_{\loc}$ if and only if $\beta>-1+\frac{p-1}{2}$. 
%--------------------------------------------------------------------
\section{Main result for the approximating solution $\ue$}
%--------------------------------------------------------------------
In this section we prove an integral inequality, similar to \eqref{eq:main2-new}, for the approximating solution $\ue$. To this end, we denote
$$ \Ve_\vp(\nabla \ue):=\frac{\vp'(\sqrtue)}{\sqrtue}\nabla\ue 
\quad\text{and}\quad 
\Ve_\psi(\nabla \ue):=\frac{\psi'(\sqrtue)}{\sqrtue}\nabla\ue. $$
We emphasize that throughout this section $0<\epsilon\leq 1$ is a fixed parameter.

\begin{proposition} \label{prop:Regularization-Main}
Let $\vp$ be a $C^2$-regular Orlicz function with $(p,q)$ growth and let $\ue\in W^{1,\vp}(U)$ be a weak solution of \eqref{eq:Regularized-Dirichlet-Problem}. Suppose that $\psi$ is another $C^2$-regular Orlicz function with $(\tilde{p},\tilde{q})$ growth and $\Ve_\psi(\nabla u):=\frac{\psi'(\sqrtue)}{\sqrtue}\nabla\ue$. If the closeness function $\theta$ of $\vp$ and $\psi$, as defined in \eqref{eq:Intro-closeness}, is bounded as follows,
\begin{equation} \label{eq:Regularization-main-closeness}
    s_\theta=\sup\{\theta(t):t>0\}<\frac{2(n-1)}{n-2},
\end{equation} 
then $\Ve_\psi(\nabla \ue)\in W^{1,2}_{\loc}(U;\Rn)$. Moreover, for all concentric balls $B_r\subset B_{2r}\Subset U$ we have the local estimate
\begin{equation} \label{eq:Regularization-main}
\begin{aligned}
    \int_{B_r}|D(\Ve_\psi(\nabla \ue))|^2dx
    &\leq 
    \frac{C}{r^2}\int_{B_{2r}}|\Ve_\psi(\nabla \ue)-(\Ve_\psi(\nabla \ue))_{B_{2r}}|^2dx \\
    &\quad 
    +C\int_{B_{2r}}\Big(\frac{\psi'(\sqrtue)}{\vp'(\sqrtue)}f\Big)^2dx,
\end{aligned}    
\end{equation}
where $C=C(n,p,q,\tilde{p},\tilde{q},s_\theta)>0$ is independent of $\epsilon$.
\end{proposition}

Proposition \ref{prop:Regularization-Main} is a version of \cite{Challal2010}*{Theorem 2.3} containing two vector fields. Note that the right hand side is finite since $\nabla \ue$ and $f$ are bounded and $\varepsilon>0$.

The proof of Proposition \ref{prop:Regularization-Main} follows relatively easily from the pointwise inequality \eqref{eq:Inequality} of Lemma \ref{lem:Inequality}, if the Orlicz functions $\vp$ and $\psi$ and the data $f$ are known to be smooth. The smoothness of $\vp$ (and $f$) implies, by the standard elliptic regularity theory, that also $\ue$ is smooth. Consequently, one can check that if \eqref{eq:Regularization-main-closeness} holds, then Lemma \ref{lem:Inequality} is applicable with
\begin{equation} \label{eq:Va-and-Vb-if-Orlicz-functions-smooth}
    V_a=\Ve_\vp(\nabla\ue)\quad\text{and}\quad V_b=\Ve_{\psi}(\nabla\ue).
\end{equation}
The integral estimate \eqref{eq:Regularization-main} then follows by integrating the pointwise inequality \eqref{eq:Inequality} against a suitable cutoff function. 

In this paper we cover also the case when $\vp$ and $\psi$ (and $f$) do not posses enough regularity for the vector fields \eqref{eq:Va-and-Vb-if-Orlicz-functions-smooth} to belong in $C^2$. We employ the standard convolution mollification to the Orlicz functions $\vp$ and $\psi$, as well as to the solution $\ue$. We denote the resulting mollifications by $\vp_\kappa$, $\psi_\kappa$ and $\ued$, for $\kappa,\delta>0$. We check that if \eqref{eq:Regularization-main-closeness} holds, then Lemma \ref{lem:Inequality} is applicable with
$$ V_a:=\Ve_{\vp_\kappa}(\nabla \ued) 
\quad\text{and}\quad
V_b:=\Ve_{\psi_\kappa}(\nabla \ued).
$$
We integrate the resulting pointwise inequality against a cutoff function, and after integration by parts we can let $\delta\to0$ and $\kappa\to0$ to achieve Proposition \ref{prop:Regularization-Main}.
%--------------------------------------------------------------------
\subsection{Mollification of Orlicz functions} \label{sec:add-approx-kappa}
%In this subsection we regularize the Orlicz functions $\vp$ and $\psi$. This yields mollified versions of $\vp$ and $\psi$, and consequently mollified version of the derived quantities $\nu$, $\mu$ and $\theta$. The corresponding approximations converge locally uniformly.
Let 
$$ \zeta(t):=
\begin{cases}
c\exp\big(\frac{1}{t^2-1}\big) &\quad\text{if }|t|<1; \\
0 &\quad\text{if }|t|\geq 1,
\end{cases}
$$
where the constant $c>0$ is fixed such that $\int_\R\zeta(t)dt=1$. This is the standard one-dimensional mollifier.

For $\kappa>0$ small, set $\zeta_\kappa(t):=\tfrac{1}{\kappa}\zeta(\tfrac{t}{\kappa})$. Let us assume that $0<\kappa<\tfrac{1}{2}\sqrt{\epsilon}$ and consider the smooth mollifications
$$ \vp_\kappa\colon[\tfrac{1}{2}\sqrt{\epsilon},\infty)\to(0,\infty)
\quad\text{and}\quad
\psi_\kappa\colon[\tfrac{1}{2}\sqrt{\epsilon},\infty)\to(0,\infty) $$
given by
$$ \vp_\kappa(t)=\int_{-\kappa}^\kappa\vp(t-s)\zeta_\kappa(s)ds
\quad\text{and}\quad
\psi_\kappa(t)=\int_{-\kappa}^\kappa\psi(t-s)\zeta_\kappa(s)ds. $$
Then
\begin{equation} \label{eq:Convergence-of-regularized-Orlicz-f}
    \vp_\kappa\xrightarrow{\kappa\to 0}\vp
    \quad\text{and}\quad
    \psi_\kappa\xrightarrow{\kappa\to 0}\psi 
    \quad\text{in $C^2([\sqrt{\epsilon},M])$ for any $M>\sqrt{\epsilon}$}.    
\end{equation}
For later use, we denote
$$ \nuk(t):=\frac{\vp_\kappa''(t)t}{\vp_\kappa'(t)} 
\quad\text{and}\quad 
\muk(t):=\frac{\psi_\kappa''(t)t}{\psi_\kappa'(t)}, $$
and
$$ \theta_\kappa(t):=\frac{\nuk(t)}{\muk(t)}. $$
These are the mollified versions of our main quantities \eqref{eq:nu-and-mu} and \eqref{eq:theta}, respectively. It follows directly from \eqref{eq:Convergence-of-regularized-Orlicz-f} that 
\begin{equation} \label{eq:Convergence-of-regualrized-growth-rates}
    \nu_\kappa\xrightarrow{\kappa\to0}\nu
    \quad\text{and}\quad
    \mu_\kappa\xrightarrow{\kappa\to0}\mu 
    \quad\text{and}\quad
    \theta_\kappa\xrightarrow{\kappa\to0}\theta
    \quad\text{in }C([\sqrt{\epsilon},M]).
\end{equation}

The following lemma says that the $(p,q)$-growth of $\vp$ implies $(\frac{1}{2}(p+1),q+1)$-growth of $\vp_\kappa$, provided that $\kappa$ is small enough. Similarly, $(\tilde{p},\tilde{q})$-growth of $\psi$ implies $(\frac{1}{2}(\tilde{p}+1),\tilde{q}+1)$-growth of $\psi_\kappa$ and the closeness of $\vp$ and $\psi$, in the sense of \eqref{eq:Regularization-main-closeness}, implies the same closeness of $\vp_\kappa$ and $\psi_\kappa$, provided that $\kappa$ is small enough.

\begin{lemma} \label{lem:Growths-for-mollified-Orlicz-functions}
Suppose that $\varphi,\psi\colon[0,\infty)\to[0,\infty)$ are Orlicz-functions such that
\begin{align}
    \begin{cases}
    0<p-1\leq \nu(t)\leq q-1<\infty \\
    0<\tilde{p}-1\leq \mu(t)\leq \tilde{q}-1<\infty \\
    \theta(t)\leq s_\theta <\frac{2(n-1)}{n-2}
    \end{cases}
\end{align}
for all $t\in[0,\infty)$.
Let $0<\epsilon<M$ be fixed.
Then we can find a small positive constant
$$ \kappa_0=\kappa_0(\epsilon,s_\theta,p,q,\tilde{p},\tilde{q},M), $$ 
such that $0<\kappa<\kappa_0$ implies that
\begin{align}
    \begin{cases}
    0<\frac{1}{2}(p-1)\leq \nu_\kappa(t)\leq q<\infty \\
    0<\frac{1}{2}(\tilde{p}-1)\leq \mu_\kappa(t)\leq \tilde{q}<\infty \\
    \theta_\kappa(t)\leq \frac{1}{2}(s_\theta+\frac{2(n-1)}{n-2})<\frac{2(n-1)}{n-2}
    \end{cases}
\end{align}
for all $t\in[\sqrt{\epsilon},M]$.
\end{lemma}

\begin{proof}
The proof follows directly from the convergences \eqref{eq:Convergence-of-regualrized-growth-rates}. %\kom{Write more. (2 sent.)}
\end{proof}
%--------------------------------------------------------------------
\subsection{Mollification of solution $\ue$} \label{sec:add-approx-delta}
%In this subsection we regularize the solution $\ue$ and check that the derivatives $\nabla\ued$ are uniformly bounded.
For the regularization of $\ue$, let
$$ \zeta(y):=
\begin{cases}
c\exp\big(\frac{1}{|y|^2-1}\big) &\quad\text{if }|y|<1; \\
0 &\quad\text{if }|y|\geq 1,
\end{cases}
$$
where the constant $c>0$ is fixed such that $\int_{\Rn}\zeta(y)dy=1$. Denote $\zeta_\delta(y):=\frac{1}{\delta^n}\zeta(\frac{y}{\delta})$. This is the standard $n$-dimensional mollifier. Let $U'\Subset U$ be a smooth subdomain and $0<\delta<\dist(U',\partial U)$. The smooth mollification of $\ue$, 
$$ \ued\colon U'\to\R, $$
is given by
$$ \ued(x):=\int_{B(0,\delta)}\ue(x-y)\zeta_\delta(y)dy. $$
Since $\ue\in W^{2,2}_\loc(U)$, \cite{Challal2010}*{Theorem 2.1}, we have
\begin{equation} \label{eq:Convergence-of-ued}
    \ued\xrightarrow{\delta\to 0} \ue
    \quad\text{in }W^{2,2}(U').    
\end{equation}
That is, $ \ued\xrightarrow{\delta\to 0} \ue$ in $W^{2,2}_{\loc}(U)$.
Moreover, if $\delta=0$, then we interpret $u^{\epsilon,0}=\ue$.

%The mollified Orlicz functions converge on any bounded interval, and hence we need a uniform bound for $\{\nabla\ued\}_\delta$ in $L^\infty_{\loc}(U)$:

\begin{lemma} \label{lem:Uniform-Gradient-Bound}
Let $\ue\in W^{1,\vp}(U)$ be a weak solution of \eqref{eq:Regularized-Dirichlet-Problem}, $U'\Subset U$ and $\ued\colon U'\to\R$ the smooth mollification of $\ue$.
If $0\leq \delta<\frac{1}{2}\dist(U',\partial U)$, then
$$ \|\nabla \ued\|_{L^{\infty}(U')}\leq 
C_1=C_1(n,p,q,\vp'(1),\|f\|_{L^\infty(U)},\|u\|_{L^\infty(U)},\dist(U',\partial U)). $$
In particular, the constant $C_1$ is independent of $\delta$ and $\epsilon.$
\end{lemma}

\begin{proof}
We estimate as follows:
\begin{align*}
    \|\nabla \ued\|_{L^{\infty}(U')}
    &\leq
    \sup_{x\in U'}\Big|\int_{B(0,\delta)}\nabla\ue(x-y)\zeta_\delta(|y|)dy\Big| \\
    &\leq 
    \sup_{x\in U'}\int_{B(0,\delta)}|\nabla\ue(x-y)|\zeta_\delta(|y|)dy \\
    &\leq 
    \|\nabla \ue \|_{L^\infty(B(U',\delta))},
\end{align*}
where $B(U',\delta):=\{x\in\Rn:\dist(x,U')\leq \delta\}$.
Now the desired estimate follows from \eqref{eq:Uniform-bound-in-C1alpha}.
\end{proof}
%--------------------------------------------------------------------
\subsection{Application of main tool} \label{ssec:application}
In this section we explain how to employ Lemma \ref{lem:Inequality} in our setting. We use the notation introduced in Section \ref{ssec:inequality}, above the statement of Lemma \ref{lem:Inequality}.

For $0<\epsilon\leq 1$ and $0<\kappa<\tfrac{1}{2}\sqrt{\epsilon}$, let
\begin{equation} \label{eq:a-and-b}
    a(t):=\frac{\vp_\kappa(\sqrt{t^2+\epsilon})}{\sqrt{t^2+\epsilon}}
    \quad\text{and}\quad
    b(t):=\frac{\psi_\kappa(\sqrt{t^2+\epsilon})}{\sqrt{t^2+\epsilon}}.
\end{equation}
Notice that $a$ and $b$ are smooth functions on $[0,\infty)$. Computations show that
$$ \alpha(t)=\frac{a'(t)t}{a(t)} = \Big(\nuk(\sqrt{t^2+\epsilon})-1\Big)\frac{t^2}{t^2+\epsilon} $$
and
$$
\beta(t)=\frac{b'(t)t}{b(t)} = \Big(\muk(\sqrt{t^2+\epsilon})-1\Big)\frac{t^2}{t^2+\epsilon}
$$
and moreover,
$$ \gamma(t)=\frac{\alpha(t)+1}{\beta(t)+1}
=\frac{\nuk(\sqrt{t^2+\epsilon})t^2+\epsilon}{\muk(\sqrt{t^2+\epsilon})t^2+\epsilon}.
$$
By analyzing these functions we find that
\begin{equation} \label{eq:alpha-bounds}
    \min\Big\{\nuk(\sqrt{t^2+\epsilon})-1,0\Big\}
    %\leq\frac{a'(t)t}{a(t)}
    \leq \alpha(t)
    \leq \max\Big\{\nuk(\sqrt{t^2+\epsilon})-1,0\Big\},
\end{equation}
and analogously
\begin{equation} \label{eq:beta-bounds}
    \min\Big\{\muk(\sqrt{t^2+\epsilon})-1,0\Big\}
    %\leq\frac{b'(t)t}{b(t)}
    \leq \beta(t)
    \leq \max\Big\{\muk(\sqrt{t^2+\epsilon})-1,0\Big\},
\end{equation}
and finally
\begin{equation} \label{eq:gamma-bounds}
    \gamma(t)
    \leq \max\Big\{\theta_\kappa(\sqrt{t^2+\epsilon}),1\Big\}.
\end{equation}
Given such $a$ and $b$, we consider the $C^2$-vector fields
$$ V_a:=\Ve_{\vp_\kappa}(\nabla \ued) 
\quad\text{and}\quad
V_b:=\Ve_{\psi_\kappa}(\nabla \ued).
$$

\begin{lemma} \label{lem:main-lemma-smooth-case}
Let $\ue\in W^{1,\vp}(U)$ be a weak solution of \eqref{eq:Regularized-Dirichlet-Problem} and $\ued$ its smooth mollification. Fix $U'\Subset U$ and let $\eta\in C^\infty_0(U')$. Then the $(p,q)$ growth of $\vp$ and $(\tilde{p},\tilde{q})$ growth of $\psi$, together with the closeness assumption
$$ s_\theta<\frac{2(n-1)}{n-2}, $$
imply that if $0<\kappa<\kappa_0$ for some 
$$ \kappa_0=\kappa_0(\epsilon,s_\theta,n,p,q,\tilde{p},\tilde{q},\vp'(1),\|f\|_{L^\infty(U)},\|u\|_{L^\infty(U)},\dist(U',\partial U)), $$ 
then
\begin{equation}
\begin{aligned}
    &c\int_{U'}|D(\Ve_{\psi_\kappa}(\nabla  \ued))|^2\eta\, dx \\
    &\leq
    -\int_{U'}\langle D(\Ve_{\psi_\kappa}(\nabla  \ued))-(\tr(D (V_{\psi_\kappa}(\nabla  \ued)))I)(\Ve_{\psi_\kappa}(\nabla  \ued)-Z),\nabla  \eta\rangle \, dx \\
    &\quad
    +C\int_{U'}\Big(\frac{\psi_{\kappa}'(\sqrt{|\nabla  \ued|^2+\epsilon})}{\vp_{\kappa}'(\sqrt{|\nabla  \ued|^2+\epsilon})}\Big)^2\big(\diverg(\Ve_{\vp_\kappa}(\nabla  \ued))\big)^2\eta\, dx.
\end{aligned}
\end{equation}
Here $Z\in\Rn$ is an arbitrary vector, and $c=c(n,p,q,\tilde p,\tilde q,s_\theta)>0$ and $C=C(n,p,q,\tilde p,\tilde q,s_\theta)>0$ are positive constants that are independent of the regularization parameters $\epsilon$, $\kappa$ and $\delta$.
\end{lemma}

\begin{proof}
Let $0<\kappa<\frac{1}{2}\sqrt{\epsilon}$ and $0<\delta<\frac{1}{2}\dist(U',\partial U)$. Then by Lemma \ref{lem:Uniform-Gradient-Bound}
$$ \|\nabla \ued\|_{L^{\infty}(U')}\leq 
C_1=C_1(n,p,q,\vp'(1),\|f\|_{L^\infty(U)},\|u\|_{L^\infty(U)},\dist(U',\partial U)). $$

The main point of the proof is to check that Lemma \ref{lem:Inequality} is applicable with $a$ and $b$ as defined in \eqref{eq:a-and-b}, 
and $u$ replaced with $\ued$. 

Firstly, by \eqref{eq:alpha-bounds} and Lemma \ref{lem:Growths-for-mollified-Orlicz-functions} (with $M=\sqrt{C_1+1}$)
\begin{equation} \label{eq:i-alpha-bound}
\begin{aligned}
    i_\alpha
    &=\inf\{\alpha(t):0\leq t\leq C_1\} \\
    &\geq\min\Big\{\inf\{\nuk(\sqrt{t^2+\epsilon})-1:0\leq t\leq C_1\},0\Big\} \\
    &\geq \min\{\tfrac{1}{2}(p-1)-1,0\}>-1,
\end{aligned}
\end{equation}
and
\begin{equation} \label{eq:s-alpha-bound}
\begin{aligned}
    s_\alpha
    &=\sup\{\alpha(t):0\leq t\leq c_1\} \\
    &\leq\max\Big\{\sup\{\nuk(\sqrt{t^2+\epsilon})-1:0\leq t\leq C_1\},0\Big\} \\
    &\leq \max\{q-1,0\}<\infty.
\end{aligned}
\end{equation}
Similarly, by using \eqref{eq:beta-bounds} and Lemma \ref{lem:Growths-for-mollified-Orlicz-functions}
\begin{equation} \label{eq:i-beta-s-beta-bounds}
    -1<i_\beta\leq s_\beta<\infty. 
\end{equation}
For the closeness condition \eqref{eq:Cordes-for-AwrtB} of Lemma \ref{lem:Inequality}, we use \eqref{eq:gamma-bounds} and Lemma \ref{lem:Growths-for-mollified-Orlicz-functions} to obtain
\begin{equation} \label{eq:s-gamma-bound}
\begin{aligned}
    s_\gamma
    &=\sup\{\gamma(t):0\leq t\leq C_1\} \\
    &\leq \max\Big\{\sup\{\theta_\kappa(\sqrt{t^2+\epsilon}):0\leq t\leq C_1\},1\Big\} \\
    &\leq \max\Big\{\tfrac{1}{2}(s_\theta+\tfrac{2(n-1)}{n-2}),1\Big\}<\frac{2(n-1)}{n-2}.
\end{aligned}
\end{equation}
We conclude that the bounds \eqref{eq:i-alpha-bound}--\eqref{eq:s-gamma-bound} imply that we can apply Lemma \ref{lem:Inequality} for the $C^2$-vector fields
$$ V_a=\Ve_{\vp_\kappa}(\nabla \ued) 
\quad\text{and}\quad
V_b=\Ve_{\psi_\kappa}(\nabla \ued) $$
and consequently
\begin{equation} \label{eq:pointwise-inequality-for-Vekappa-ued}
\begin{aligned}
    c|D(\Ve_{\psi_\kappa}(\nabla \ued)|^2
    &\leq
    \diverg((D(\Ve_{\psi_\kappa}(\nabla \ued))-\tr(D(\Ve_{\psi_\kappa}(\nabla \ued))I)(\Ve_{\psi_\kappa}(\nabla \ued)-Z)) \\
    &\quad
    +C\Big(\frac{\psi_{\kappa}'(\sqrt{|\nabla  \ued|^2+\epsilon})}{\vp_{\kappa}'(\sqrt{|\nabla  \ued|^2+\epsilon})}\Big)^2\big(\diverg(\Ve_{\vp_\kappa}(\nabla \ued))\big)^2
\end{aligned}    
\end{equation}
everywhere in $U'$. Here $Z\in\Rn$.

Let $\eta\in C^\infty_0(U')$. We multiply the above inequality \eqref{eq:pointwise-inequality-for-Vekappa-ued} with $\eta$ and integrate over $U'$ to obtain that
\begin{equation} \label{eq:integrated-inequality-for-Vekappa-ued}
\begin{aligned}
    &c\int_{U'}|D(\Ve_{\psi_\kappa}(\nabla \ued)|^2\eta dx \\
    &\leq
    \int_{U'}\diverg((D(\Ve_{\psi_\kappa}(\nabla \ued))-\tr(D(\Ve_{\psi_\kappa}(\nabla \ued))I)(\Ve_{\psi_\kappa}(\nabla \ued)-Z))\eta dx \\
    &\quad
    +C\int_{U'}\Big(\frac{\psi_{\kappa}'(\sqrt{|\nabla  \ued|^2+\epsilon})}{\vp_{\kappa}'(\sqrt{|\nabla  \ued|^2+\epsilon})}\Big)^2\big(\diverg(\Ve_{\vp_\kappa}(\nabla \ued))\big)^2\eta dx \\
    &=
    -\int_{U'}\la(D(\Ve_{\psi_\kappa}(\nabla \ued))-\tr(D(\Ve_{\psi_\kappa}(\nabla \ued))I)(\Ve_{\psi_\kappa}(\nabla \ued)-Z),\nabla\eta\ra dx \\
    &\quad
    +C\int_{U'}\Big(\frac{\psi_{\kappa}'(\sqrt{|\nabla  \ued|^2+\epsilon})}{\vp_{\kappa}'(\sqrt{|\nabla  \ued|^2+\epsilon})}\Big)^2\big(\diverg(\Ve_{\vp_\kappa}(\nabla \ued))\big)^2\eta dx.
\end{aligned}
\end{equation}
In the equality on the fourth row of the display \eqref{eq:integrated-inequality-for-Vekappa-ued} we applied integration by parts. The proof is finished.
\end{proof}
%--------------------------------------------------------------------
\subsection{Convergence of the additional approximations: Let $\delta\to0$ and $\kappa\to0$}

Let $\kappa_0$ denote the upper bound for $\kappa$, given by Lemma \ref{lem:main-lemma-smooth-case}.

\begin{lemma} \label{lem:Convergence}
Let $0<\epsilon\leq 1$ and $0\leq \kappa<\kappa_0$.
We have $\Ve_{\vp_\kappa}(\nabla\ue)\in W^{1,2}_{\loc}(U)$ and $\Ve_{\psi_\kappa}(\nabla\ue)\in W^{1,2}_{\loc}(U)$, and their derivatives are given by the formulas
$$
D(\Ve_{\vp_\kappa}(\nabla\ue))
=
\frac{\vp'_\kappa(\sqrt{|\nabla\ue|^2+\epsilon})}{\sqrt{|\nabla\ue|^2+\epsilon}} 
\Big(
I+\Big(\nu_\kappa\Big(\sqrt{|\nabla\ue|^2+\epsilon}\Big)-1\Big)\frac{\nabla\ue\otimes \nabla\ue}{|\nabla\ue|^2+\epsilon}
\Big)D^2\ue $$
and
$$
D(\Ve_{\psi_\kappa}(\nabla\ue))
=
\frac{\psi'_\kappa(\sqrt{|\nabla\ue|^2+\epsilon})}{\sqrt{|\nabla\ue|^2+\epsilon}} 
\Big(
I+\Big(\mu_\kappa\Big(\sqrt{|\nabla\ue|^2+\epsilon}\Big)-1\Big)\frac{\nabla\ue\otimes \nabla\ue}{|\nabla\ue|^2+\epsilon}
\Big)D^2\ue, $$
where $I$ denotes the identity matrix.
Moreover,
$$ \Ve_{\vp_\kappa}(\nabla\ued)
\xrightarrow{\delta\to0} \Ve_{\vp_\kappa}(\nabla\ue)
\quad\text{and}\quad 
\Ve_{\psi_\kappa}(\nabla\ued)
\xrightarrow{\delta\to0} \Ve_{\psi_\kappa}(\nabla\ue) 
\quad\text{in }W^{1,2}_{\loc}(U). $$
\end{lemma}

\begin{remark}
The above Lemma \ref{lem:Convergence} holds also in the non-mollified case $\kappa=0$.
\end{remark}

\begin{proof}[Proof of Lemma \ref{lem:Convergence}]
We only consider the case with $\vp_\kappa$ because the case with $\psi_\kappa$ is identical.

In order to ease the notation below, let us set for small  $\delta\geq 0$
$$ g_\delta:=\frac{\vp'_\kappa(\sqrt{|\nabla\ued|^2+\epsilon})}{\sqrt{|\nabla\ued|^2+\epsilon}} $$
and
\begin{align*}
    G_\delta
    :=\frac{\vp'_\kappa(\sqrt{|\nabla\ued|^2+\epsilon})}{\sqrt{|\nabla\ued|^2+\epsilon}} 
    \Big(
    I+\Big(\nu_\kappa\Big(\sqrt{|\nabla\ued|^2+\epsilon}\Big)-1\Big)\frac{\nabla\ued\otimes \nabla\ued}{|\nabla\ued|^2+\epsilon}
    \Big).
\end{align*}
Moreover, let $g:=g_0$ and $G:=G_0$. Note that $g_\delta$ is a scalar valued function and
\begin{align*}
    \Ve_{\vp_\kappa}(\nabla\ued)=g_\delta \nabla\ued
\end{align*}
for all $\delta\geq 0$, and $G_\delta$ is a matrix-valued function and
\begin{align*}
    D(\Ve_{\vp_\kappa}(\nabla\ued))
    =G_\delta D^2\ued
\end{align*}
for all $\delta>0$. In case $\delta=0$, we aim to show that  $\Ve_{\vp_\kappa}(\nabla\ue)$ is weakly differentiable and $D(\Ve_{\vp_\kappa}(\nabla\ue))=GD^2\ue$. 

We first observe the following bounds for $g_\delta$ and $G_\delta$. Namely, for any $U'\Subset U$ and for any $\delta\geq 0$
\begin{equation}
    \|g_\delta\|_{L^\infty(U')}
    \leq
    \frac{\big|\vp'_\kappa\Big(\sqrt{\|\nabla\ued\|_{L^\infty(U')}^2+\epsilon}\Big)\big|}{\sqrt{\epsilon}},
\end{equation}
and, by Lemma \ref{lem:Growths-for-mollified-Orlicz-functions},
\begin{equation}
    \|G_\delta\|_{L^\infty(U')}
    \leq
    \frac{C(n,p,q)\big|\vp'_\kappa\Big(\sqrt{\|\nabla\ued\|_{L^\infty(U')}^2+\epsilon}\Big)\big|}{\sqrt{\epsilon}}.
\end{equation}
By employing Lemma \ref{lem:Uniform-Gradient-Bound} we conclude that for any $\delta\geq0$
\begin{equation} \label{eq:Sup-bound-for-f-and-F}
    \|g_\delta\|_{L^\infty(U')}
    +
    \|G_\delta\|_{L^\infty(U')}
    \leq \frac{C}{\sqrt{\epsilon}}
\end{equation}
where
$$ C=C(n,p,q,\vp'(1),\vp'(M),\|f\|_{L^\infty(U)},\|u\|_{L^\infty(U)},\dist(U',\partial U)) $$ 
with some
$$ M=M(n,p,q,\vp'(1),\|f\|_{L^\infty(U)},\|u\|_{L^\infty(U)},\dist(U',\partial U)). $$ 
In particular $C$ is independent of $\delta$.

The bound \eqref{eq:Sup-bound-for-f-and-F} implies that $\Ve_{\vp_\kappa}(\nabla\ued)\in L^2_{\loc}(U)$ and $G_\delta \nabla\ued\in L^2_{\loc}(U)$ for any $\delta\geq 0$. Indeed, for any $U'\Subset U$ we can use the estimate \eqref{eq:Sup-bound-for-f-and-F} to find
\begin{align*}
    \int_{U'}|\Ve_{\vp_\kappa}(\nabla\ued)|^2dx
    =
    \int_{U'}|g_\delta \nabla\ued|^2dx 
    \leq
    \frac{C}{\epsilon}\int_{U'}|\nabla\ued|^2dx
\end{align*}
and
\begin{align*}
    \int_{U'}|G_\delta D^2\ued|^2dx
    &\leq
    \frac{C}{\epsilon}\int_{U'}|D^2\ued|^2dx.
\end{align*}

Now we check that $g_\delta \nabla\ued\xrightarrow{\delta\to0}g\nabla\ue$ and
$G_\delta D^2\ued\xrightarrow{\delta\to0}G D^2\ue$
in $L^2_{\loc}(U)$. This also follows from the bound \eqref{eq:Sup-bound-for-f-and-F}. Indeed, for any $U'\Subset U$
\begin{equation} \label{eq:Conv-1}
\begin{aligned}
    \int_{U'}|g_\delta \nabla\ued-g\nabla\ue|^2dx
    &\leq 4\Big\{
    \int_{U'}|g_\delta|^2|\nabla\ued-\nabla\ue|^2dx
    +\int_{U'}|g_\delta-g|^2|\nabla\ue|^2dx\Big\} \\
    &\leq 4\Big\{
    \frac{C}{\epsilon}\int_{U'}|\nabla\ued-\nabla\ue|^2dx
    +\int_{U'}|g_\delta-g|^2|\nabla\ue|^2dx\Big\} \\
    &\quad\xrightarrow{\delta\to0} 0.
\end{aligned}
\end{equation}
The convergence of the first integral follows trivially from the convergence $\ued\xrightarrow{\delta\to0}\ue$ in $W^{2,2}_{\loc}(U)$. For the convergence of the second integral, we additionally employ dominated convergence theorem.
Similarly,
\begin{equation} \label{eq:Conv-2}
\begin{aligned}
    \int_{U'}|G_\delta D^2\ued-GD^2\ue|^2dx
    &\leq 4\Big\{
    \int_{U'}|G_\delta|^2|D^2\ued-D^2\ue|^2dx
    +\int_{U'}|G_\delta-G|^2|D^2\ue|^2dx\Big\} \\
    &\leq 4\Big\{
    \frac{C}{\epsilon}\int_{U'}|D^2\ued-D^2\ue|^2dx
    +\int_{U'}|G_\delta-G|^2|D^2\ue|^2dx\Big\} \\
    &\quad\xrightarrow{\delta\to0} 0.
\end{aligned}    
\end{equation}
The two convergences \eqref{eq:Conv-1} and \eqref{eq:Conv-2} imply that we can let $\delta\to0$ in the integration by parts formula
$$ \int_U \Big(\Ve_{\vp_\kappa}(\nabla\ued)\Big)_i\frac{\partial\eta}{\partial x_j}dx
= -\int_U \frac{\partial}{\partial x_j}\Big(\Ve_{\vp_\kappa}(\nabla\ued)\Big)_i\eta dx
\quad\text{for all }\eta\in C^\infty_0(U) $$ 
for any $i,j=1,\ldots,n$ to conclude that $\Ve_{\vp_\kappa}(\nabla\ue)$ is weakly differentiable and $D(\Ve_{\vp_\kappa}(\nabla\ue))=GD^2\ue$. This finishes the proof.
\end{proof}

\begin{corollary} \label{cor:Pointwise-PDE}
In particular $\diverg\big(\Ve_{\vp}(\nabla\ue)\big)=f$ almost everywhere in $U$.
\end{corollary}

\begin{proof}
By the weak formulation of the PDE in \eqref{eq:Regularized-Dirichlet-Problem},
\begin{equation} \label{eq:weak-form-of-regularization}
    \int_U \la \Ve_{\vp}(\nabla \ue),\nabla \eta\ra dx = \int_U f\eta dx
\end{equation}
for all $\eta\in\CcU$. Since $\Ve_\vp(\nabla \ue) \in W^{1,2}_{\loc}(U)$ we can rewrite \eqref{eq:weak-form-of-regularization} as
$$ \int_U \left(\dv (\Ve_\vp(\nabla \ue))- f\right) \eta \, dx =0. $$
Fundamental theorem of calculus of variations yields that $\dv\big(\Ve_\vp(\nabla \ue)\big)=f$ almost everywhere in $U$.
\end{proof}

\begin{lemma} \label{lem:main-lemma-delta0-kappa0}
Let $\ue\in W^{1,\vp}(U)$ be a weak solution of \eqref{eq:Regularized-Dirichlet-Problem}. Fix $U'\Subset U$ and let $\eta\in C^\infty_0(U')$. Then the $(p,q)$ growth of $\vp$ and $(\tilde{p},\tilde{q})$ growth of $\psi$, together with the closeness assumption
$$ s_\theta<\frac{2(n-1)}{n-2}, $$
imply that 
\begin{equation}
\begin{aligned}
    &c\int_{U'}|D(\Ve_{\psi}(\nabla \ue))|^2\eta\, dx \\
    &\leq
    -\int_{U'}\langle (D(\Ve_{\psi}(\nabla \ue))-\tr(D(\Ve_{\psi}(\nabla \ue))I)(\Ve_{\psi}(\nabla \ue)-Z),\nabla \eta\rangle \, dx \\
    &\quad
    +C\int_{U'}\Big(\frac{\psi'(\sqrt{|\nabla \ue|^2+\epsilon})}{\vp'(\sqrt{|\nabla \ue|^2+\epsilon})}\Big)^2\big(\diverg(\Ve_{\vp}(\nabla \ue))\big)^2\eta\, dx.
\end{aligned}
\end{equation}
Here $Z\in\Rn$ is an arbitrary vector, and $c=c(n,p,q,\tilde p,\tilde q,s_\theta)>0$ and $C=C(n,p,q,\tilde p,\tilde q,s_\theta)>0$ are positive constants that are independent of $\epsilon$.
\end{lemma}

\begin{proof}
Suppose that $0<\kappa<\kappa_0$ and $0<\delta<\frac{1}{2}\dist(U',U)$.
By Lemma \ref{lem:main-lemma-smooth-case} we have
\begin{equation}
\begin{aligned}
    &c\int_{U'}|D (\Ve_{\psi_\kappa}(\nabla \ued))|^2\eta\, dx \\
    &\leq
    -\int_{U'}\langle D(\Ve_{\psi_\kappa}(\nabla \ued))-(\tr(D( V_{\psi_\kappa}(\nabla \ued)))I)(\Ve_{\psi_\kappa}(\nabla \ued)-Z),\nabla \eta\rangle \, dx \\
    &\quad
    +C\int_{U'}\Big(\frac{\psi_{\kappa}'(\sqrt{|\nabla \ued|^2+\epsilon})}{\vp_{\kappa}'(\sqrt{|\nabla \ued|^2+\epsilon})}\Big)^2\big(\diverg(\Ve_{\vp_\kappa}(\nabla \ued))\big)^2\eta\, dx.
\end{aligned}
\end{equation}
We employ Lemma \ref{lem:Convergence} and dominated convergence theorem to let $\delta\to0$ in the above estimate. More precisely, for the convergence of the left hand side and the first integral on the right hand side, Lemma \ref{lem:Convergence} is sufficient. For the convergence of the second integral on the right hand side, we additionally note that 
$$ \Big(\frac{\psi_{\kappa}'(\sqrt{|\nabla \ued|^2+\epsilon})}{\vp_{\kappa}'(\sqrt{|\nabla \ued|^2+\epsilon})}\Big)^2
\leq \Big(\frac{\psi_{\kappa}'(\sqrt{C_1^2+1})}{\vp_{\kappa}'(\sqrt{\epsilon})}\Big)^2=:C_2 \quad\text{in }U' $$
for every $\delta\geq 0$. Here we employed Lemma \ref{lem:Uniform-Gradient-Bound}. Thus
\begin{align*}
    &\Big|
    \int_{U'}\Big(\frac{\psi_{\kappa}'(\sqrt{|\nabla \ued|^2+\epsilon})}{\vp_{\kappa}'(\sqrt{|\nabla \ued|^2+\epsilon})}\Big)^2\big(\diverg(\Ve_{\vp_\kappa}(\nabla \ued))\big)^2\eta\, dx 
    -
    \int_{U'}\Big(\frac{\psi_{\kappa}'(\sqrt{|\nabla \ue|^2+\epsilon})}{\vp_{\kappa}'(\sqrt{|\nabla \ue|^2+\epsilon})}\Big)^2\big(\diverg(\Ve_{\vp_\kappa}(\nabla \ue))\big)^2\eta\, dx\Big| \\
    %&\leq
    %\Big|
    %\int_{U'}\Big(\frac{\psi_{\kappa}'(\sqrt{|\nabla \ued|^2+\epsilon})}{\vp_{\kappa}'(\sqrt{|\nabla \ued|^2+\epsilon})}\Big)^2
    %\Big\{\big(\diverg(\Ve_{\vp_\kappa}(\nabla \ued))\big)^2-\big(\diverg(\Ve_{\vp_\kappa}(\nabla \ue))\big)^2\Big\}\eta\,dx\Big| \\ 
    %&\quad
    %+\Big|
    %\int_{U'}\Big\{\Big(\frac{\psi_{\kappa}'(\sqrt{|\nabla \ued|^2+\epsilon})}{\vp_{\kappa}'(\sqrt{|\nabla \ued|^2+\epsilon})}\Big)^2
    %-\Big(\frac{\psi_{\kappa}'(\sqrt{|\nabla \ue|^2+\epsilon})}{\vp_{\kappa}'(\sqrt{|\nabla \ue|^2+\epsilon})}\Big)^2\Big\}\big(\diverg(\Ve_{\vp_\kappa}(\nabla \ue))\big)^2\eta\, dx \Big| \\
    &\leq
    C_2
    \int_{U'}\Big|\big(\diverg(\Ve_{\vp_\kappa}(\nabla \ued))\big)^2-\big(\diverg(\Ve_{\vp_\kappa}(\nabla \ue))\big)^2\Big||\eta|\,dx \\ 
    &\quad
    +
    \int_{U'}\Big|\Big(\frac{\psi_{\kappa}'(\sqrt{|\nabla \ued|^2+\epsilon})}{\vp_{\kappa}'(\sqrt{|\nabla \ued|^2+\epsilon})}\Big)^2
    -\Big(\frac{\psi_{\kappa}'(\sqrt{|\nabla \ue|^2+\epsilon})}{\vp_{\kappa}'(\sqrt{|\nabla \ue|^2+\epsilon})}\Big)^2\Big|\big(\diverg(\Ve_{\vp_\kappa}(\nabla \ue))\big)^2|\eta|\, dx.    
\end{align*}
The integral on the second row of the above display tends to zero as $\delta\to 0$ by Lemma \ref{lem:Convergence}. For the integral on the third row, we pass to a subsequence -- still denoted by itself -- such that $\nabla\ued\xrightarrow{\delta\to0}\nabla\ue$ almost everywhere in $U'$. Dominated convergence theorem allows us to conclude that the integral converges to zero as $\delta\to0$.

We obtain
\begin{equation}\label{eq:kappa-epsilon}
\begin{aligned}
    &c\int_{U'}|D (\Ve_{\psi_\kappa}(\nabla \ue))|^2\eta\, dx \\
    &\leq
    -\int_{U'}\langle D (\Ve_{\psi_\kappa}(\nabla \ue))-(\tr(D (V_{\psi_\kappa}(\nabla \ue)))I)(\Ve_{\psi_\kappa}(\nabla \ue)-Z),\nabla \eta\rangle \, dx \\
    &\quad
    +C\int_{U'}\Big(\frac{\psi_{\kappa}'(\sqrt{|\nabla \ue|^2+\epsilon})}{\vp_{\kappa}'(\sqrt{|\nabla \ue|^2+\epsilon})}\Big)^2\big(\diverg(\Ve_{\vp_\kappa}(\nabla \ue))\big)^2\eta\, dx.
\end{aligned}
\end{equation}
Finally, the derivative formulas of Lemma \ref{lem:Convergence} allow us to see that the integrands in \eqref{eq:kappa-epsilon} converge uniformly as $\kappa\to0$. Consequently we can let $\kappa\to0$ in \eqref{eq:kappa-epsilon} to arrive at the desired estimate.
\end{proof}

This Lemma, together with Corollary \ref{cor:Pointwise-PDE}, implies Proposition \ref{prop:Regularization-Main}.

\begin{proof}[Proof of Proposition \ref{prop:Regularization-Main}]
Let us fix a ball $B_{2r}\Subset U$ and select a nonnegative cutoff function $\eta\in C^\infty_0(B_{2r})$ such that
\begin{equation} \label{eq:cutoff}
    \eta\equiv 1 \enskip\text{in }B_r
    \quad\text{and}\quad
    |\nabla\eta| \leq \frac{C}{r}.
\end{equation}
By Lemma \ref{lem:main-lemma-delta0-kappa0} any Corollary \ref{cor:Pointwise-PDE} we obtain that
\begin{equation}
\begin{aligned}
    &c\int_{U'}|D(\Ve_{\psi}(\nabla \ue))|^2\eta^2\, dx \\
    &\leq
    -2\int_{U'}\langle D(\Ve_{\psi}(\nabla \ue))-(\tr(D(V_{\psi}(\nabla \ue)))I)(\Ve_{\psi}(\nabla \ue)-Z),\nabla \eta\rangle\eta \, dx \\
    &\quad
    +C\int_{U'}\Big(\frac{\psi'(\sqrt{|\nabla \ue|^2+\epsilon})}{\vp'(\sqrt{|\nabla \ue|^2+\epsilon})}f\Big)^2\eta^2\, dx
\end{aligned}
\end{equation}
for any vector $Z\in\Rn$. 
An application of Young's inequality yields that
\begin{equation}
\begin{aligned}
    \frac{c}{2}\int_{U'}|D(\Ve_{\psi}(\nabla \ue))|^2\eta^2\, dx 
    &\leq
    C\int_{U'}|\Ve_{\psi}(\nabla \ue)-Z|^2|\nabla \eta|^2\, dx \\
    &\quad
    +C\int_{U'}\Big(\frac{\psi'(\sqrt{|\nabla \ue|^2+\epsilon})}{\vp'(\sqrt{|\nabla \ue|^2+\epsilon})}f\Big)^2\eta^2\, dx.
\end{aligned}
\end{equation}
Select $Z=(\Ve_\psi(\nabla\ue))_{B_{2r}}$ and employ \eqref{eq:cutoff} to arrive at the desired estimate.
\end{proof}
%--------------------------------------------------------------------
\section{Proofs of main results}
%--------------------------------------------------------------------
We let $\epsilon\to$ in Proposition \ref{prop:Regularization-Main} to recover our main results.
%--------------------------------------------------------------------
\subsection{Let $\epsilon\to0$ in homogeneous case} \label{ssec:epsilon-to-0-H}

\begin{proof}[Proof of Theorem \ref{thm:main1}]
Let us fix the concentric balls $B_r\Subset B_{2r}\Subset \Omega$. Select a smooth subdomain $U\Subset \Omega$ such that $B_{2r}\Subset U$.
Consider the Dirichlet problem
\begin{equation}  \label{eq:regularized-Orlicz-Laplacian-eq-2}
\begin{cases}
\begin{aligned}
-\dv \Big(\frac{\varphi'\big(\sqrt{|\nabla \ue |^2+\varepsilon}\big)}{\sqrt{|\nabla \ue |^2+\varepsilon}}\nabla \ue\Big)=0 &\quad\text{in }U; \\
\ue=u &\quad\text{on }\partial U,
\end{aligned}
\end{cases}
\end{equation}
for some $\varepsilon>0$.
 
By Proposition \ref{prop:Regularization-Main} with $f=0$ we have that $V_\psi^\varepsilon(\nabla \ue) \in W^{1,2}_{\loc}(U)$ and
\begin{equation} \label{eq:Regularized-estimate}
    \int_{B_{r}} |D( \Ve_\psi(\nabla \ue ))|^2 \, dx \leq \frac{C}{r^2} \int_{B_{2r}} |\Ve_\psi(\nabla \ue ) - (\Ve_\psi(\nabla \ue ))_{B_{2r}}|^2\, dx
\end{equation}
where $C=C(n,p,q,\tilde{p},\tilde{q},s_\theta)>0$ is independent of $\epsilon$. Our goal is to let $\varepsilon\to 0$ in \eqref{eq:Regularized-estimate}. 

First, \eqref{eq:Uniform-bound-in-C1alpha} and estimate \eqref{eq:Regularized-estimate} imply that $\{\Ve_\psi(\nabla\ue)\}_\epsilon$ is bounded in $W^{1,2}(B_r)$. Indeed, from \eqref{eq:Uniform-bound-in-C1alpha} we conclude the pointwise estimate
\begin{equation} \label{eq:pointwise-estimate-proof-of-main1}
\begin{aligned}
        |\Ve_{\psi}(\nabla\ue)|^2 
        &=
        \Big|\frac{\psi'(\sqrtue)}{\sqrtue}\Big|^2|\nabla\ue|^2 \\
        &\leq
        \Big|\psi'\Big(\sqrtue\Big)\Big|^2 \\
        &\leq 
        \Big|\psi'\Big(\sqrt{\|\nabla\ue\|_{L^\infty(B_r)}+\epsilon}\Big)\Big|^2 \\
        &\leq
        \big|\psi'(C+1)\big|^2,
\end{aligned} 
\end{equation}
where 
$$ C=C(n,p,q,\vp'(1),\|u\|_{L^\infty(U)},\dist(B_r,\partial U)), $$
and hence, by \eqref{eq:Regularized-estimate},
\begin{align*}
        \int_{B_r}|D(\Ve_{\vp}(\nabla\ue))|^2dx
        &\leq
        \frac{C}{r^2}\int_{B_{2r}}|\Ve_\psi(\nabla\ue)-(\Ve_\psi(\nabla\ue))_{B_{2r}}|^2dx \leq Cr^{n-2}.
\end{align*}
Thus we see that $\|V_\psi^\varepsilon(\nabla \ue)\|_{W^{1,2}(B_r)} \leq C$.

By compactness arguments, there exists a function $V \in W^{1,2}(B_r)$ such that ${V_\psi^\varepsilon(\nabla \ue) \to V}$ weakly in $W^{1,2}(B_r)$ and strongly in $L^2(B_r)$ when $\varepsilon \to 0$. Especially there exists a subsequence converging pointwise almost everywhere. Since by the convergence \eqref{eq:Convergence-of-ue-in-C1} $\nabla \ue \to \nabla u$ in $C(B_r)$, we see that $V = V_\psi(\nabla u)$ almost everywhere in $B_r$. 

Combining pointwise convergence and the uniform boundedness \eqref{eq:pointwise-estimate-proof-of-main1} for Lebesgue's dominated convergence theorem we see that
\begin{align*}
\int_{B_{2r}} | V_\psi^\varepsilon(\ue)-( V_\psi^\varepsilon(\ue))_{B_{2r}}|^2 \, dx \xrightarrow{\epsilon\to0} \int_{B_{2r}} |V_\psi(\nabla u)-(V_\psi(\nabla u ))_{B_{2r}}|^2 \, dx.
\end{align*}
Finally using weak lower semicontinuity of the norm, we have
\begin{align*}
    \int_{B_{r}} |D( V_\psi(\nabla\ue))|^2 \, dx &\leq \liminf_{\varepsilon \to 0} \int_{B_{r}} |D( \Ve_\psi(\nabla \ue) )|^2 \, dx \\
    &\leq \liminf_{\varepsilon \to 0} \frac{C}{r^2} \int_{B_{2r}} |\Ve_\psi(\nabla \ue ) - (\Ve_\psi(\nabla \ue ))_{B_{2r}}|^2\, dx \\
    &\leq \dfrac{C}{r^2} \int_{B_{2r}} |V_\psi(\nabla u) - (V_\psi(\nabla u ))_{B_{2r}}|^2 \, dx.
\end{align*}
The proof is finished.
\end{proof}
%--------------------------------------------------------------------
\subsection{Let $\epsilon\to0$ in the non-homogeneous case} \label{ssec:epsilon-to-0-NH}

\begin{lemma} \label{lem:pointwise-bound-for-ratio}
Let $\vp$ and $\psi$ be $C^2$ regular Orlicz functions with $(p,q)$ growth and $(\tilde{p},\tilde{q})$ growth, respectively.
Suppose that the ratio function $\rho$ of $\vp$ and $\psi$, as defined in \eqref{eq:Intro-ratio-def}, is bounded as follows,
$$ s_\rho=\sup\{\rho(t):0\leq t\leq1\}<\infty. $$
Then
$$ \frac{\psi'(t)}{\vp'(t)} \leq C(1+\psi'(t)) \quad\text{for all }t\geq0. $$
Here $C=\max\{s_\rho,(\vp'(1))^{-1}\}$.
\end{lemma}

\begin{proof} 
For any $t\geq0$,
\begin{align*}
    \frac{\psi'(t)}{\vp'(t)}
    &=
    \chi_{\{0\leq t\leq 1\}}(t)\frac{\psi'(t)}{\vp'(t)}+\chi_{\{t>1\}}(t)\frac{\psi'(t)}{\vp'(t)} \\
    &\leq
    \sup\Big\{\frac{\psi'(t)}{\vp'(t)}:0\leq t\leq1\Big\}+\frac{\psi'(t)}{\vp'(1)} \\
    &\leq C(1+\psi'(t)).
\end{align*}
Above $\chi$ denotes the indicator function.
\end{proof}

\begin{proof}[Proof of Theorem \ref{thm:main2-new}]
Let us fix the concentric balls $B_r\Subset B_{2r}\Subset \Omega$. Select a smooth subdomain $U\Subset \Omega$ such that $B_{2r}\Subset U$.
Consider the Dirichlet problem
\begin{equation}  \label{eq:regularized-Orlicz-Laplacian-eq-3}
\begin{cases}
\begin{aligned}
-\dv \Big(\frac{\varphi'\big(\sqrt{|\nabla \ue |^2+\varepsilon}\big)}{\sqrt{|\nabla \ue |^2+\varepsilon}}\nabla \ue\Big)=f &\quad\text{in }U; \\
\ue=u &\quad\text{on }\partial U,
\end{aligned}
\end{cases}
\end{equation}
for some $\varepsilon>0$.
 
By Proposition \ref{prop:Regularization-Main} we have that $V_\psi^\varepsilon(\nabla \ue) \in W^{1,2}_{\loc}(U;\Rn)$ and
\begin{equation} \label{eq:Regularized-estimate-with-f-original}
\begin{aligned}
    \int_{B_{r}} |D( \Ve_\psi(\nabla \ue ))|^2 \, dx &\leq \frac{C}{r^2} \int_{B_{2r}} |\Ve_\psi(\nabla \ue ) - (\Ve_\psi(\nabla \ue ))_{B_{2r}}|^2\, dx
    +C\int_{B_{2r}}\Big(\frac{\psi'(\sqrtue)}{\vp'(\sqrtue)}f\Big)^2dx
\end{aligned}    
\end{equation}
where $C=C(n,p,q,\tilde{p},\tilde{q},s_\theta)>0$ is independent of $\epsilon$. Our goal is to let $\varepsilon\to 0$ in \eqref{eq:Regularized-estimate-with-f-original}. We proceed similarly as in the proof of Theorem \ref{thm:main1}.

Let us first check that $\{\Ve_\psi(\nabla\ue)\}_\epsilon$ is bounded in $W^{1,2}(B_r)$. As in the proof of Theorem \ref{thm:main1}, \eqref{eq:Uniform-bound-in-C1alpha} yields the pointwise estimate
\begin{equation} \label{eq:pointwise-estimate-in-proof-of-main2-new}
\begin{aligned}
        |\Ve_{\psi}(\nabla\ue)|^2 
        &\leq
        \big|\psi'(\sqrt{C+1})\big|^2,
\end{aligned} 
\end{equation}
where 
$$ C=C(n,p,q,\vp'(1),\|f\|_{L^\infty(U)},\|u\|_{L^\infty(U)},\dist(B_{2r},\partial U)). $$
On the other hand, by Lemma \ref{lem:pointwise-bound-for-ratio},
\begin{align*}
    \frac{\psi'(\sqrtue)}{\vp'(\sqrtue)}
    &\leq
    C\Big(1+\psi'\big(\sqrtue\big)\Big)
    \leq 
    C\Big(1+\psi'\big(\sqrt{C+1}\big)\Big)
\end{align*}
where $C=\max\{s_\rho,(\vp'(1))^{-1}\}$.
Hence,
\begin{align*}
    \int_{B_{r}} |D (\Ve_\psi(\nabla \ue))|^2 \, dx 
    &\leq Cr^{n-2}
\end{align*}
for some constant $C$ that is independent of $\epsilon$. Thus we see that $\|V_\psi^\varepsilon(\nabla \ue)\|_{W^{1,2}(B_r)} \leq C$.

By compactness arguments, there exists a function $V \in W^{1,2}(B_r)$ such that ${V_\psi^\varepsilon(\nabla \ue) \to V}$ weakly in $W^{1,2}(B_r)$ and strongly in $L^2(B_r)$ when $\varepsilon \to 0$. Especially there exists a subsequence converging pointwise almost everywhere. Since by the convergence \eqref{eq:Convergence-of-ue-in-C1} $\nabla \ue \to \nabla u$ in $C(B_r)$, we see that $V = V_\psi(\nabla u)$ almost everywhere in $B_r$. 

By the weak lower semicontinuity of norm we get from \eqref{eq:Regularized-estimate-with-f-original}
\begin{align*}
        \int_{B_r}|D(V_{\psi}(\nabla u))|^2dx
        &\leq 
        \liminf_{\epsilon\to0}
        \int_{B_r}|D(\Ve_{\psi}(\nabla\ue))|^2dx \\
        &\leq
        \liminf_{\epsilon\to0}
        \Big(\frac{C}{r^2}\int_{B_{2r}}|\Ve_\psi(\nabla\ue)-(\Ve_\psi(\nabla\ue))_{B_{2r}}|^2dx \\
        &\quad+C\int_{B_{2r}}\Big(\frac{\psi'(\sqrtue)}{\vp'(\sqrtue)}f\Big)^2dx \Big)
\end{align*}
By the same argument as in proof of Theorem \ref{thm:main1}
\begin{align*}
        \liminf_{\epsilon\to0}
        \int_{B_{2r}}|\Ve_\psi(\nabla\ue)-(\Ve_\psi(\nabla\ue))_{B_{2r}}|^2dx 
        =
        \int_{B_{2r}}|V_\psi(\nabla u)-(V_\psi(\nabla u))_{B_{2r}}|^2dx.
\end{align*}    
By the continuity of $\vp'$ and $\psi'$, together with the convergence  $\ue \to u$ in $C^{1}_{\loc}(U)$, see \eqref{eq:Convergence-of-ue-in-C1}, we see that
$$ \frac{\psi'(\sqrtue)}{\vp'(\sqrtue)}\xrightarrow{\epsilon\to0}
\frac{\psi'(|\nabla u|)}{\vp'(|\nabla u|)} $$
everywhere in $U$. Combining the pointwise convergence and the uniform boundedness \eqref{eq:pointwise-estimate-in-proof-of-main2-new} for Lebesgue's dominated convergence theorem we see that
\begin{align*}
        \lim_{\epsilon\to0}
        \int_{B_{2r}}\Big(\frac{\psi'(\sqrtue)}{\vp'(\sqrtue)}f\Big)^2 dx
        =
        \int_{B_{2r}}\Big(\frac{\psi'(|\nabla u|)}{\vp'(|\nabla u|)}f\Big)^2 dx
\end{align*}   
Hence, we conclude the desired inequality.
\end{proof}
%--------------------------------------------------------------------
\subsection{Higher integrability}

An essential tool for obtaining higher integrability is the following lemma, due to Gehring \cite{Gehring1973}. 

\begin{lemma}[Gehring's lemma, \cites{Gehring1973,Giusti2003}]
Let $g \in L^1(B_{2R})$ be non-negative. Assume that there exists $s\in(0, 1)$ such that for every ball $B_R$, with $B_r \subset B_{2r} \Subset B_{2R}$, holds
\[
\fint_{B_r} g\, dx 
\leq \gamma\bigg(\Big(\fint_{B_{2r}} g^s\, dx\Big)^{\frac{1}{s}} + \fint_{B_{2r}} h \, dx \bigg)
\]
with $0<s<1$. Assume that  the function $h$ belongs to $L^t(B_{2R})$ for some $t>1$. Then there exists $\ell=\ell(n,s,\gamma,t)>1$ such that
\[
\bigg(\fint_{B_R} g^{\ell} \, dx \bigg)^{1/\ell} \leq C\bigg( \fint_{B_{2R}} g \, dx + \Big(\fint_{B_{2R}} h^\ell \, dx\Big)^{1/\ell}\bigg)
\]
where $C=C(n,s,\gamma,t)>0$.
\end{lemma}

\begin{proof}[Proof of Corollary \ref{thm_higher-integrability-1}]
Combining Theorem~\ref{thm:main1} and Sobolev--Poincar'e inequality we have
\begin{align*}
    \left( \fint_{B_r} |D(V_\psi (\nabla u))|^2 \, dx \right)^{1/2} &\leq \dfrac{C}{r} \left( \fint_{B_{2r}} |V_\psi(\nabla u) - (V_\psi(\nabla u))_{B_{2r}}|^2 \, dx \right)^{1/2} \\
    &\leq C \left(\fint_{B_{2r}} |D(V_\psi(\nabla u))|^{\frac{2n}{n+2}} \, dx \right)^{\frac{n+2}{2n}}.
\end{align*}
We apply Gehring's lemma with
$$ g=|D(V_\psi(\nabla u))|^2 \quad\text{and}\quad h=0 $$   
to obtain the desired result.
\end{proof}

\begin{proof}[Proof of Corollary \ref{thm_higher-integrability-2-new}]
Combining Theorem~\ref{thm:main2-new} and Sobolev--Poincar\'e inequality we have
\begin{align*}
    \left( \fint_{B_r} |D(V_\psi (\nabla u))|^2 \, dx \right)^{1/2} &\leq \dfrac{C}{r} \left( \fint_{B_{2r}} |V_\psi(\nabla u) - (V_\psi(\nabla u))_{B_{2r}}|^2 \, dx \right)^{1/2} + C \left(\fint_{B_{2r}} \Big(\frac{\psi'(|\nabla u|)}{\vp'(|\nabla u|)}f\Big)^2 \, dx \right)^{1/2} \\
    &\leq C \left(\fint_{B_{2r}} |D(V_\psi(\nabla u))|^{\frac{2n}{n+2}} \, dx \right)^{\frac{n+2}{2n}} +  C \left(\fint_{B_{2r}} \Big(\frac{\psi'(|\nabla u|)}{\vp'(|\nabla u|)}f\Big)^2 dx \right)^{1/2}.
\end{align*}
We apply Gehring's lemma with
$$ g=|D(V_\psi(\nabla u))|^2 \quad\text{and}\quad h=\Big(\frac{\psi'(|\nabla u|)}{\vp'(|\nabla u|)}f\Big)^2  $$  to obtain the desired result.
Note that $h\in L^\infty_{\loc}(\Om)$ due to the assumptions of Theorem \ref{thm:main2-new}, and by Lemma~\ref{lem:pointwise-bound-for-ratio} and \eqref{eq:u-is-C1}. 
\end{proof}
%--------------------------------------------------------------------
\section*{Acknowledgements}

Part of this work was done when S. Sarsa was visiting University of Warsaw. She wants to thank the department there for their kind hospitality during her stay.

The authors wish to thank Iwona Chlebicka and Petteri Harjulehto for reading the manuscript and giving feedback.
%--------------------------------------------------------------------
\bibliographystyle{amsplain}
\bibliography{bibliography}
%--------------------------------------------------------------------
\end{document}